\def\year{2022}\relax
\documentclass[letterpaper]{article}
\usepackage{aaai22}
\usepackage{times}
\usepackage{helvet}
\usepackage{courier}
\usepackage[hyphens]{url}
\usepackage{graphicx}
\usepackage{natbib}
\usepackage{caption}
\DeclareCaptionStyle{ruled}{labelfont=normalfont,labelsep=colon,strut=off}
\usepackage{mathtools}
\usepackage{amssymb}
\usepackage{amsthm}
\usepackage{array}
\usepackage{booktabs}
\usepackage[caption=false,font=footnotesize,labelformat=simple]{subfig}

\usepackage{algorithm}
\usepackage{algpseudocode}

\newcommand*{\RR}{\mathbb R}
\newcommand*{\ee}{\mathrm e}
\newcommand*{\dd}{\mathrm d}
\DeclareMathOperator{\st}{s.t.}
\DeclareMathOperator{\rank}{rank}
\DeclareMathOperator{\dist}{dist}
\DeclareMathOperator{\tr}{tr}
\DeclareMathOperator{\vect}{vec}
\DeclareMathOperator{\mat}{mat}
\DeclareMathOperator{\RIP}{RIP}
\DeclarePairedDelimiter{\abs}{\lvert}{\rvert}
\DeclarePairedDelimiter{\norm}{\lVert}{\rVert}
\DeclareMathOperator{\OPT}{OPT}
\DeclareMathOperator{\LMI}{LMI}

\newtheorem{theorem}{Theorem}
\newtheorem{lemma}[theorem]{Lemma}
\theoremstyle{definition}
\newtheorem{definition}{Definition}
\newtheorem{assumption}{Assumption}

\title{Local and Global Linear Convergence of General Low-rank Matrix Recovery Problems}
\author {
Yingjie Bi,\textsuperscript{\rm 1}
Haixiang Zhang, \textsuperscript{\rm 2}
Javad Lavaei \textsuperscript{\rm 1}
}
\affiliations {
\textsuperscript{\rm 1} Industrial Engineering and Operations Research, University of California, Berkeley \\
\textsuperscript{\rm 2} Department of Mathematics, University of California, Berkeley \\
yingjiebi@berkeley.edu, haixiang\_zhang@berkeley.edu, lavaei@berkeley.edu
}

\begin{document}
\maketitle

\begin{abstract}
We study the convergence rate of gradient-based local search methods for solving low-rank matrix recovery problems with general objectives in both symmetric and asymmetric cases, under the assumption of the restricted isometry property. First, we develop a new technique to verify the Polyak--{\L}ojasiewicz inequality in a neighborhood of the global minimizers, which leads to a local linear convergence region for the gradient descent method. Second, based on the local convergence result and a sharp strict saddle property proven in this paper, we present two new conditions that guarantee the global linear convergence of the perturbed gradient descent method. The developed local and global convergence results provide much stronger theoretical guarantees than the existing results. As a by-product, this work significantly improves the existing bounds on the RIP constant required to guarantee the non-existence of spurious solutions.
\end{abstract}

\section{Introduction}\label{sec:intro}

The low-rank matrix recovery problem is to recover an unknown low-rank ground truth matrix from certain measurements. This problem has a variety of applications in machine learning, such as recommendation systems \citep{KBV2009} and motion detection \citep{ZYY2013,FS2020}, and in engineering problems, such as power system state estimation \citep{ZML2018}.

In this paper, we consider two variants of the low-rank matrix recovery problem with a general measurement model represented by an arbitrary smooth function. The first variant is the symmetric problem, in which the ground truth $M^* \in \RR^{n \times n}$ is a symmetric and positive semidefinite matrix with $\rank(M^*)=r$, and $M^*$ is a global minimizer of some loss function $f_s$. Then, $M^*$ can be recovered by solving the optimization problem:
\begin{equation}\label{eq:symrecorig}
\begin{aligned}
\min \quad & f_s(M) \\
\st \quad & \rank(M) \leq r, \\
& M \succeq 0, \: M \in \RR^{n \times n}.
\end{aligned}
\end{equation}
Note that minimizing $f_s(M)$ over positive semidefinite matrices without the rank constraint would often lead to finding a solution with the highest-rank possible rather than the rank-constrained solution $M^*$. The second variant of the low-rank matrix recovery problem to be studied is the asymmetric problem, in which $M^* \in \RR^{n \times m}$ is a possibly non-square matrix with $\rank(M^*)=r$, and it is a global minimizer of some loss function $f_a$. Similarly, $M^*$ can be recovered by solving
\begin{equation}\label{eq:asymrecorig}
\begin{aligned}
\min \quad & f_a(M) \\
\st \quad & \rank(M) \leq r, \\
& M \in \RR^{n \times m}.
\end{aligned}
\end{equation}
As a special case, the loss function $f_s$ or $f_a$ can be induced by linear measurements. In this situation, we are given a linear operator $\mathcal A:\RR^{n \times n} \to \RR^p$ or $\mathcal A:\RR^{n \times m} \to \RR^p$, where $p$ denotes the number of measurements. To recover $M^*$ from the vector $d=\mathcal A(M^*)$, the function $f_s(M)$ or $f_a(M)$ is often chosen to be
\begin{equation}\label{eq:linearobj}
\frac{1}{2}\norm{\mathcal A(M)-d}^2.
\end{equation}
Besides, there are many natural choices for the loss function, such as a nonlinear model associated with the 1-bit matrix recovery \citep{DPBW2014}.

The symmetric problem \eqref{eq:symrecorig} can be transformed into an unconstrained optimization problem by factoring $M$ as $XX^T$ with $X \in \RR^{n \times r}$, which leads to the following equivalent formulation:
\begin{equation}\label{eq:symrec}
\min_{X \in \RR^{n \times r}}f_s(XX^T).
\end{equation}
In the asymmetric case, one can similarly factor $M$ as $UV^T$ with $U \in \RR^{n \times r}$ and $V \in \RR^{m \times r}$. Note that $(UP,V(P^{-1})^T)$ gives another possible factorization of $M$ for any invertible matrix $P \in \RR^{r \times r}$. To reduce the redundancy, a regularization term is usually added to the objective function to enforce that the factorization is balanced, i.e., $U^TU=V^TV$ is satisfied \citep{TBSS2016}. Since every factorization can be converted into a balanced one by selecting an appropriate $P$, the original asymmetric problem \eqref{eq:asymrecorig} is equivalent to
\begin{equation}\label{eq:asymrec}
\min_{U \in \RR^{n \times r},V \in \RR^{m \times r}}f_a(UV^T)+\frac{\phi}{4}\norm{U^TU-V^TV}_F^2,
\end{equation}
where $\phi>0$ is an arbitrary constant.

To handle the symmetric and asymmetric problems in a unified way, we will use the same notation $X$ to denote the matrix of decision variables in both cases. In the symmetric case, $X$ is obtained from the equation $M=XX^T$. In the asymmetric case, $X$ is defined as
\[
X=\begin{bmatrix}
U \\
V
\end{bmatrix} \in \RR^{(n+m) \times r}.
\]
To rewrite the asymmetric problem \eqref{eq:asymrec} in terms of $X$, we apply the technique in \citet{TBSS2016} by defining an auxiliary function $F:\RR^{(n+m) \times (n+m)} \to \RR$ as
\begin{multline}\label{eq:asymtransfunc}
F\left(\begin{bmatrix}
N_{11} & N_{12} \\
N_{21} & N_{22}
\end{bmatrix}\right)=\frac{1}{2}(f_a(N_{12})+f_a(N_{21}^T)) \\
+\frac{\phi}{4}(\norm{N_{11}}_F^2+\norm{N_{22}}_F^2-\norm{N_{12}}_F^2-\norm{N_{21}}_F^2),
\end{multline}
in which the argument of the function $F$ is partitioned into four blocks, denoted as $N_{11} \in \RR^{n \times n}$, $N_{12} \in \RR^{n \times m}$, $N_{21} \in \RR^{m \times n}$, $N_{22} \in \RR^{m \times m}$. The problem \eqref{eq:asymrec} then reduces to
\begin{equation}\label{eq:asymtrans}
\min_{X \in \RR^{(n+m) \times r}}F(XX^T),
\end{equation}
which is a special case of the symmetric problem \eqref{eq:symrec}. Henceforth, the objective functions of the two problems will be referred as to $g_s(X)=f_s(XX^T)$ and $g_a(X)=F(XX^T)$, respectively.

The unconstrained problems \eqref{eq:symrec} and \eqref{eq:asymrec} are often solved by local search algorithms, such as the gradient descent method, due to their efficiency in handling large-scale problems. Since the objective functions $g_s(X)$ and $g_a(X)$ are nonconvex, local search methods may converge to a spurious (non-global) local minimum. To guarantee the absence of such spurious solutions, the restricted isometry property (RIP) defined below is the most common condition imposed on the functions $f_s$ and $f_a$ \citep{BNS2016,GJZ2017,ZLTW2018,ZWYG2018,ZJSL2018,ZSL2019,HLB2020,ZZ2020,BL2020-pb,ZBL2021-p,Zhang2021-p}.

\begin{definition}[\citet{RFP2010,ZLTW2018}]\label{def:rip}
A twice continuously differentiable function $f_s:\RR^{n \times n} \to \RR$ satisfies the \emph{restricted isometry property} of rank $(2r_1,2r_2)$ for a constant $\delta \in [0,1)$, denoted as $\delta$-$\RIP_{2r_1,2r_2}$, if
\[
(1-\delta)\norm{K}_F^2 \leq [\nabla^2f_s(M)](K,K) \leq (1+\delta)\norm{K}_F^2
\]
holds for all matrices $M,K \in \RR^{n \times n}$ with $\rank(M) \leq 2r_1$ and $\rank(K) \leq 2r_2$. In the case when $r_1=r_2=r$, the notation $\RIP_{2r,2r}$ will be simplified as $\RIP_{2r}$. A similar definition can be also made for the asymmetric loss function $f_a$.
\end{definition}

The state-of-the-art results on the non-existence of spurious local minima are presented in \citet{ZBL2021-p,Zhang2021-p}. \citet{ZBL2021-p} shows that the problem \eqref{eq:symrec} or \eqref{eq:asymrec} is devoid of spurious local minima if i) the associated function $f_s$ or $f_a$ satisfies the $\delta$-$\RIP_2$ property with $\delta<1/2$ in case $r=1$, ii) the function $f_s$ or $f_a$ satisfies the $\delta$-$\RIP_{2r}$ property with $\delta \leq 1/3$ in case $r>1$. \citet{Zhang2021-p} further shows that a special case of the symmetric problem \eqref{eq:symrec} does not have spurious local minima if iii) $f_s$ is in the form \eqref{eq:linearobj} given by linear measurements and satisfies the $\delta$-$\RIP_{2r}$ property with $\delta<1/2$. The absence of spurious local minima under the above conditions does not automatically imply the existence of numerical algorithms with a fast convergence to the ground truth. In this paper, we will significantly strengthen the above two results by establishing the global linear convergence for problems with an arbitrary rank $r$ and a general loss function $f_s$ or $f_a$ under the almost same RIP assumption as above, i.e., $\delta<1/2$ for the symmetric problem and $\delta<1/3$ for the asymmetric problem.

One common approach to establish fast convergence is to first show that the objective function has favorable regularity properties, such as strong convexity, in a neighborhood of the global minimizers, which guarantees that common iterative algorithms will converge to a global minimizer at least linearly if they are initialized in this neighborhood. Second, given the local convergence result, certain algorithms can be utilized to reach the above neighborhood from an arbitrary initial point. Note that randomization and stochasticity are often needed in those algorithms to avoid saddle points that are far from the ground truth, such as random initialization \citep{LSJR2016} or random perturbation during the iterations \citep{GHJY2015,JGNK2017}. In this paper, we deal with the two above-mentioned aspects for the low-rank matrix recovery problem separately.

\subsection{Notations and Conventions}

In this paper, $I_n$ denotes the identity matrix of size $n \times n$, $A \otimes B$ denotes the Kronecker product of matrices $A$ and $B$, and $A \succeq 0$ means that $A$ is a symmetric and positive semidefinite matrix. $\sigma_i(A)$ denotes the $i$-th largest singular value of the matrix $A$. $\mathbf A=\vect(A)$ is the vector obtained from stacking the columns of a matrix $A$. For a vector $\mathbf A$ of dimension $n^2$, its symmetric matricization $\mat_S(\mathbf A)$ is defined as $(A+A^T)/2$ with $A$ being the unique matrix satisfying $\mathbf A=\vect(A)$. For two matrices $A$ and $B$ of the same size, their inner product is denoted as $\langle A,B\rangle=\tr(A^TB)$ and $\norm{A}_F=\sqrt{\langle A,A\rangle}$ denotes the Frobenius norm of $A$. Given a matrix $M$ and a set $\mathcal Z$ of matrices, define
\[
\dist(X,\mathcal Z)=\min_{Z \in \mathcal Z}\norm{X-Z}_F.
\]
Moreover, $\norm{v}$ denotes the Euclidean norm of a vector $v$. The action of the Hessian $\nabla^2f(M)$ of a matrix function $f$ on any two matrices $K$ and $L$ is given by
\[
[\nabla^2f(M)](K,L)=\sum_{i,j,k,l}\frac{\partial^2f}{\partial M_{ij}\partial M_{kl}}(M)K_{ij}L_{kl}.
\]

\subsection{Summary of Main Contributions}

For the local convergence, we prove in Section~\ref{sec:local} that a regularity property named the Polyak--{\L}ojasiewicz (PL) inequality always holds in a neighborhood of the global minimizers. The PL inequality is significantly weaker than the regularity condition used in previous works to study the local convergence of the low-rank matrix recovery problem, while it still guarantees a linear convergence to the ground truth. Hence, as will be compared with the prior literature in Section~\ref{sec:related}, not only are the obtained local regularity regions remarkably larger than the existing ones, but also they require significantly weaker RIP assumptions. Specifically, if $f_s$ satisfies the $\delta$-$\RIP_{2r}$ property for an arbitrary $\delta$, we will show that there exists some constant $\mu>0$ such that the objective function $g_s$ of the symmetric problem \eqref{eq:symrec} satisfies the PL inequality
\[
\frac{1}{2}\norm{\nabla g_s(X)}_F^2 \geq \mu(g_s(X)-f_s(M^*))
\]
for all $X$ in the region
\[
\{X \in \RR^{n \times r} | \dist(X,\mathcal Z) \leq \tilde C\}
\]
with
\[
\tilde C<\sqrt{2(\sqrt 2-1)}\sqrt{1-\delta^2}\sigma_r(M^*)^{1/2}.
\]
Here, $\dist(X,\mathcal Z)$ is the Frobenius distance between the matrix $X$ and the set $\mathcal Z$ of global minimizers of the problem \eqref{eq:symrec}. A similar result will also be derived for the asymmetric problem \eqref{eq:asymrec}. Based on these results, local linear convergence can then be established. Compared with the previous results, our new results are advantageous for two reasons. First, the weaker RIP assumptions imposed by our results allow them to be applicable to a much broader class of problems, especially those problems with nonlinear measurements where the RIP constant of the loss function $f_s$ or $f_a$ varies at different points. In this case, the region in which the RIP constant is below the previous bounds may be significantly small or even empty, while the region satisfying our bounds is much larger since the radius of the region is increased by more than a constant factor. Second, when the RIP constant is large and global convergence cannot be established due to the existence of spurious solutions, the enlarged local regularity regions identified in this work can reduce the sample complexity to find the correct initial point converging to the ground truth. This has a major practical value in problems like data analytics in power systems \citep{JLSB2021} in which there is a fundamental limit to the number of measurements due to the physics of the network.

For the global convergence analysis, in Section~\ref{sec:global}, we first study the symmetric problem \eqref{eq:symrec} with an arbitrary objective and an arbitrary rank $r$ and prove that the objective $g_s$ satisfies the strict saddle property if the function $f_s$ has the $\delta$-$\RIP_{2r}$ property with $\delta<1/2$. Note that this result is sharp, because in \citet{ZJSL2018} a counterexample has been found that contains spurious local minima under $\delta=1/2$. Using the above strict saddle property and the local convergence result proven in Section~\ref{sec:local}, we show that the perturbed gradient descent method with local improvement will find an approximate solution $X$ satisfying $\norm{XX^T-M^*}_F \leq \epsilon$ in $O(\log 1/\epsilon)$ number of iterations for an arbitrary tolerance $\epsilon$. Moreover, the convergence result for symmetric problems also implies the global linear convergence for asymmetric problems under the $\delta$-$\RIP_{2r}$ condition with $\delta<1/3$.

\subsection{Assumptions}

The assumptions required in this work will be introduced below. To avoid using different notations for the symmetric and asymmetric problems, we use the universal notation $f(M)$ henceforth to denote either $f_s(M)$ or $f_a(M)$. Similarly, $M^*$ denotes the ground truth in either of the cases.

\begin{assumption}\label{ass:smooth}
The function $f$ is twice continuously differentiable. In addition, its gradient $\nabla f$ is $\rho_1$-restricted Lipschitz continuous for some constant $\rho_1$, i.e., the inequality
\[
\norm{\nabla f(M)-\nabla f(M')}_F \leq \rho_1\norm{M-M'}_F
\]
holds for all matrices $M$ and $M'$ with $\rank(M) \leq r$ and $\rank(M') \leq r$. The Hessian of the function $f$ is also $\rho_2$-restricted Lipschitz continuous for some constant $\rho_2$, i.e., the inequality
\[
\abs{[\nabla^2f(M)-\nabla^2f(M')](K,K)} \leq \rho_2\norm{M-M'}_F\norm{K}_F^2
\]
holds for all matrices $M,M',K$ with $\rank(M) \leq r$, $\rank(M') \leq r$ and $\rank(K) \leq 2r$.
\end{assumption}
\begin{assumption}\label{ass:rip}
The function $f$ satisfies the $\delta$-$\RIP_{2r}$ property. Furthermore, $\rho_1$ in Assumption~\ref{ass:smooth} is chosen to be large enough such that $\rho_1 \geq 1+2\delta$.
\end{assumption}
\begin{assumption}
The ground truth $M^*$ satisfies $\norm{M^*}_F \leq D$, and the initial point $X_0$ of the local search algorithm also satisfies $\norm{X_0X_0^T}_F \leq D$, where $D$ is a constant given by the prior knowledge (every large enough $D$ satisfies this assumption).
\end{assumption}
\begin{assumption}\label{ass:regcoeff}
In the asymmetric problem \eqref{eq:asymrec}, the coefficient $\phi$ of the regularization term is chosen to be $\phi=(1-\delta)/2$.
\end{assumption}

Note that the results of this paper still hold if the gradient and Hessian of the function $f$ are restricted Lipschitz continuous only over a bounded region. Here, for simplicity we assume that these properties hold for all low-rank matrices.

As mentioned before Definition~\ref{def:rip}, the RIP-related Assumption~\ref{ass:rip} is a widely used assumption in studying the landscape of low-rank matrix recovery problems, which is satisfied in a variety of problems, such as those for which $f$ is given by a sufficiently large number of random Gaussian linear measurements \citep{CP2011}. Moreover, in the case when the function $f$ does not satisfy the RIP assumption globally, it often satisfies RIP in a neighborhood of the global minimizers, and the theorems in this paper can still be applied to obtain local convergence results.

For the asymmetric problem, it can be verified that, by choosing the coefficient $\phi$ of the regularization term as in Assumption~\ref{ass:regcoeff}, the function $F$ in \eqref{eq:asymtrans} satisfies the $2\delta/(1+\delta)$-$\RIP_{2r}$ property after scaling (see \citet{ZBL2021-p}). Other values of $\phi$ can also lead to the RIP property on $F$, but the specific value in Assumption~\ref{ass:regcoeff} is the one minimizing the RIP constant. Furthermore, if $M^*=U^*V^{*T}$ is a balanced factorization of the ground truth $M^*$, then
\begin{equation}\label{eq:auggroundtruth}
\tilde M^*=\begin{bmatrix}
U^* \\
V^*
\end{bmatrix}\begin{bmatrix}
U^{*T} & V^{*T}
\end{bmatrix} \in \RR^{(n+m) \times (n+m)}
\end{equation}
is called the augmented ground truth, which is obviously a global minimizer of the transformed asymmetric problem \eqref{eq:asymtrans}. $\tilde M^*$ is independent of the factorization $(U^*,V^*)$, and
\[
\norm{\tilde M^*}_F=2\norm{M^*}_F \leq 2D, \quad \sigma_r(\tilde M^*)=2\sigma_r(M^*).
\]
We include the proofs of the above statements in Appendix~\ref{app:factor} for completeness. In addition, we prove in Appendix~\ref{app:factor} that the gradient and Hessian of the function $g_s$ in the symmetric problem \eqref{eq:symrec} and those of the function $g_a$ in the asymmetric problem \eqref{eq:asymrec} share the same Lipschitz property over a bounded region. Using the above observations, one can translate any results developed for symmetric problems to similar results for asymmetric problems by simply replacing $\delta$ with $2\delta/(1+\delta)$, $D$ with $2D$, and $\sigma_r(M^*)$ with $2\sigma_r(M^*)$.

\section{Related Works}\label{sec:related}

The low-rank matrix recovery problem has been investigated in numerous papers. In this section, we focus on the existing results related to the linear convergence for the factored problems \eqref{eq:symrec} and \eqref{eq:asymrec} solved by local search methods.

The major previous results on the local regularity property are summarized in Table~\ref{tab:localreg}. In this table, each number in the last column reported for the existing works denotes the radius $R$ such that their respective objective functions $g$ satisfy the $(\alpha,\beta)$-regularity condition
\[
\langle\nabla g(X),X-\mathcal P_{\mathcal Z}(X)\rangle \geq \frac{\alpha}{2}\dist(X,\mathcal Z)^2+\frac{1}{2\beta}\norm{\nabla g(X)}_F^2
\]
for all matrices $X$ with $\dist(X,\mathcal Z) \leq R$. Here, $\mathcal Z$ is the set of global minimizers, and $\mathcal P_{\mathcal Z}(X)$ is a global minimizer $Z \in \mathcal Z$ that is the closest to $X$. The $(\alpha,\beta)$-regularity condition is slightly weaker than the strong convexity condition, and it can lead to linear convergence on the same region. In Table~\ref{tab:localreg}, we do not include specialized results that are only applicable to a specific objective \citep{JGNK2017,HLZ2020-p}, or probabilistic results for randomly generated measurements \citep{ZL2015}. Moreover, \citet{LL2020,ZCG2020} used the accelerated gradient descent to obtain a faster convergence rate, but their convergence regions are even smaller than the ones based on the $(\alpha,\beta)$-regularity condition as listed in Table~\ref{tab:localreg}. Each number in the last column reported for our results refers to the radius of the region satisfying the PL inequality, which is a weaker condition than the $(\alpha,\beta)$-regularity condition while offering the same convergence rate guarantee. It can be observed that we have identified far larger regions than the existing ones under weaker RIP assumptions by replacing the $(\alpha,\beta)$-regularity condition with the PL inequality.

\begin{table*}
\centering
\renewcommand*{\arraystretch}{1.85}
\begin{tabular}{cc>{\centering\arraybackslash}m{2cm}c}
\toprule
Paper & Objective & Assumption & Radius of Local Regularity Region \\\midrule
\citet{BKS2016} & S/G & $f_s$ Convex, \newline $\delta_{2r} \leq \delta$ & $\dfrac{1}{100}\dfrac{1-\delta}{1+\delta}\dfrac{\sigma_r(M^*)}{\sigma_1(M^*)}\sigma_r(M^*)^{1/2}$ \\
\citet{TBSS2016} & S/L & $\delta_{6r} \leq 1/10$ & $\dfrac{1}{4}\sigma_r(M^*)^{1/2}$ \\
\citet{TBSS2016} & A/L & $\delta_{6r} \leq 1/25$ & $\dfrac{1}{4}\sigma_r(M^*)^{1/2}$ \\
\citet{PKCS2018} & A/G & $f_a$ Convex, \newline $\delta_{2r} \leq \delta$ & $\dfrac{\sqrt 2}{10}\sqrt{\dfrac{1-\delta}{1+\delta}}\sigma_r(M^*)^{1/2}$ \\
\citet{ZLTW2021} & A/G & $\delta_{2r,4r} \leq 1/50$ & $\sigma_r(M^*)^{1/2}$ \\\midrule
Ours & S/G & $\delta_{2r} \leq \delta$ & $0.91\sqrt{1-\delta^2}\sigma_r(M^*)^{1/2}$ \\
Ours & A/G & $\delta_{2r} \leq \delta$ & $1.29\dfrac{\sqrt{1+2\delta-3\delta^2}}{1+\delta}\sigma_r(M^*)^{1/2}$ \\\bottomrule
\end{tabular}
\caption{Previous local regularity results for the low-rank matrix recovery problems and the comparison with our results (``S'', ``A'', ``L'' and ``G'' stand for the symmetric case, asymmetric case, linear measurement and general nonlinear function).}\label{tab:localreg}
\end{table*}

Regarding the existing global convergence results for the low-rank matrix recovery problem, \citet{JGNK2017} established the global linear convergence for the symmetric problem in a very specialized case with $f_s$ being a quadratic loss function, and \citet{TBSS2016} proposed the Procrustes flow method with the global linear convergence for the linear measurement case under the assumption that the function $f_s$ satisfies the $1/10$-$\RIP_{6r}$ property for symmetric problems or the function $f_a$ satisfies the $1/25$-$\RIP_{6r}$ property for asymmetric problems under a careful initialization. \citet{ZWL2015} established the global linear convergence for asymmetric problems with linear measurements under the assumption that $f_a$ satisfies $\delta$-$\RIP_{2r}$ with $\delta \leq O(1/r)$ using alternating exact minimization over variables $U$ and $V$ in \eqref{eq:asymrec}. In addition, the strict saddle property proven in \citet{GJZ2017} leads to the global linear convergence of perturbed gradient methods for the linear measurement case under the $1/10$-$\RIP_{2r}$ assumption for symmetric problems and the $1/20$-$\RIP_{2r}$ assumption for asymmetric problems. Later, \citet{ZLTW2018} proved a weaker strict saddle property under the $1/5$-$\RIP_{2r,4r}$ assumption for symmetric problems with general objectives, while \citet{LZT2017-p} proved the same weaker property under the $1/5$-$\RIP_{2r,4r}$ assumption for asymmetric problems with general objectives and nuclear norm regularization. Our results requiring the $\delta$-$\RIP_{2r}$ property with $\delta<1/2$ for symmetric problems with general objectives and the $\delta$-$\RIP_{2r}$ property with $\delta<1/3$ for asymmetric problems with general objectives depend on significantly weaker RIP assumptions and thus can be applied to a broader class of problems, which is a major improvement over all previous results on the global linear convergence.

Besides local search methods for the factored problems, there are other approaches for tackling the low-rank matrix recovery. Earlier works such as \citet{CR2009,RFP2010} solved the original nonconvex problems based on convex relaxations. Although they can achieve good performance guarantees under the RIP assumptions, they are not suitable for large-scale problems. Other approaches for solving the low-rank matrix recovery include applying the inertial proximal gradient descent method directly to the original objective functions without factoring the decision variable $M$ \citep{DHLR2020}. However, it may converge to an arbitrary critical point, while in this paper we show that RIP-based local search methods can guarantee the global convergence to a global minimum.

\section{Local Convergence}\label{sec:local}

In this section, we present the local regularity results for problems \eqref{eq:symrec} and \eqref{eq:asymrec}, which state that the functions $g_s$ and $g_a$ satisfy the PL inequality locally, leading to local linear convergence results for the gradient descent method. The proofs are delegated to Appendix~\ref{app:localproof}.

First, we consider the symmetric problem \eqref{eq:symrec}. The development of the local PL inequality for this problem is enlightened by the high-level idea behind the proof of the absence of spurious local minima in \citet{ZSL2019,ZZ2020,BL2020-pb}. The objective is to find a function $f_s^*$ corresponding to the worst-case scenario, meaning that it satisfies the $\delta$-$\RIP_{2r}$ property with the smallest possible $\delta$ while the PL inequality is violated at a particular matrix $X$. This is achieved by designing a semidefinite program parameterized by $X$ with constraints implied by the $\delta$-$\RIP_{2r}$ property and the negation of the PL inequality. Denote the optimal value of the semidefinite program by $\delta_f^*(X)$. If a given function $f_s$ satisfies $\delta$-$\RIP_{2r}$ with $\delta<\delta_f^*(X)$ for all $X \in \RR^{n \times r}$ in a neighborhood of the global minimizers, it can be concluded that the PL inequality holds for all matrices in this neighborhood.

\begin{lemma}\label{lem:localpl}
Consider the symmetric problem \eqref{eq:symrec} and an arbitrary positive number $\tilde C$ satisfying
\begin{equation}\label{eq:plradius}
\tilde C<\sqrt{2(\sqrt 2-1)}\sqrt{1-\delta^2}\sigma_r(M^*)^{1/2}.
\end{equation}
There exists a constant $\mu>0$ such that the PL inequality
\[
\frac{1}{2}\norm{\nabla g_s(X)}_F^2 \geq \mu(g_s(X)-f_s(M^*))
\]
holds for all matrices in the region
\begin{equation}\label{eq:plneighbor}
\{X \in \RR^{n \times r} | \dist(X,\mathcal Z) \leq \tilde C\},
\end{equation}
where $\mathcal Z$ is the set of global minimizers of the problem \eqref{eq:symrec}.
\end{lemma}

In the above, note that $\sigma_r(M^*)$ and thus $\tilde C$ are always positive because $M^*$ is assumed to be rank $r$. Both the $(\alpha,\beta)$-regularity condition used in the prior literature and the PL inequality deployed here guarantee a linear convergence if it is already known that the trajectory at all iterations remains within the region in which the associated condition holds. However, there is a key difference between these two conditions. The $(\alpha,\beta)$-regularity condition ensures that $\dist(X,\mathcal Z)$ is nonincreasing during the iterations under a sufficiently small step size, and thus the trajectory never leaves the local neighborhood. In contrast, the weaker PL inequality may not be able to guarantee this property. To resolve this issue, in our convergence proof we will adopt a different distance function given by $\norm{XX^T-M^*}_F$. By Taylor's formula and the definition of the $\delta$-$\RIP_{2r}$ property, we have
\begin{equation}\label{eq:globalrip}
\begin{aligned}
\frac{1-\delta}{2}\norm{M-M^*}_F^2 &\leq f_s(M)-f_s(M^*) \\
&\leq \frac{1+\delta}{2}\norm{M-M^*}_F^2,
\end{aligned}
\end{equation}
for all matrices $M \in \RR^{n \times n}$ with $\rank(M) \leq r$. Therefore, if $M,M' \in \RR^{n \times n}$ are two matrices such that $f_s(M) \leq f_s(M')$, then the inequality \eqref{eq:globalrip} implies that
\begin{equation}\label{eq:levelset}
\norm{M-M^*}_F \leq \sqrt\frac{1+\delta}{1-\delta}\norm{M'-M^*}_F.
\end{equation}
Therefore, the distance function $\norm{XX^T-M^*}_F$ is almost nonincreasing if the function value $g_s(X)$ does not increase. Combining this idea with the local PL inequality proved in Lemma~\ref{lem:localpl}, we obtain the next local convergence result.

\begin{theorem}\label{thm:localconvsym}
For the symmetric problem \eqref{eq:symrec}, the gradient descent method converges to the optimal solution linearly if the initial point $X_0$ satisfies
\[
\norm{X_0X_0^T-M^*}_F<2(\sqrt 2-1)(1-\delta)\sigma_r(M^*)
\]
and the step size $\eta$ satisfies
\[
1/\eta \geq 12\rho_1r^{1/2}\left(\sqrt\frac{1+\delta}{1-\delta}\norm{X_0X_0^T-M^*}_F+D\right).
\]
Specifically, there exists some constant $\mu>0$ (which depends on $X_0$ but not on $\eta$) such that
\begin{multline}\label{eq:convrate}
\norm{X_tX_t^T-M^*}_F \leq (1-\mu\eta)^{t/2}\sqrt\frac{1+\delta}{1-\delta}\norm{X_0X_0^T-M^*}_F, \\
\forall t \in \{0,1,\dots\},
\end{multline}
where $X_t$ denotes the output of the algorithm at iteration $t$.
\end{theorem}

Note that since the left-hand side of \eqref{eq:convrate} is nonnegative, we have $0 \leq 1-\mu\eta \leq 1$. As a remark, although our bound on the step size $\eta$ in Theorem~\ref{thm:localconvsym} seems complex, it essentially says that $\eta$ needs to be small, and the upper bound on the acceptable values of the step size can be explicitly calculated out routinely after all the parameters of the problem are given. Furthermore, using the transformation from asymmetric problems to symmetric problems, one can obtain parallel results for the asymmetric problem \eqref{eq:asymrec} as below.

\begin{theorem}\label{thm:localconvasym}
Consider the asymmetric problem \eqref{eq:asymrec}. The PL inequality is satisfied in the region
\[
\{X \in \RR^{(n+m) \times r} | \dist(X,\mathcal Z) \leq \tilde C\},
\]
where $\mathcal Z$ denotes the set of global minimizers and
\[
\tilde C<2\sqrt{\sqrt 2-1}\frac{\sqrt{1+2\delta-3\delta^2}}{1+\delta}\sigma_r(M^*)^{1/2}.
\]
Moreover, local linear convergence is guaranteed for the gradient descent method if the initial point $X_0$ satisfies
\[
\norm{X_0X_0^T-\tilde M^*}_F<4(\sqrt 2-1)\frac{1-\delta}{1+\delta}\sigma_r(M^*)
\]
and the step size $\eta$ satisfies
\[
1/\eta \geq 12\rho_1r^{1/2}\left(\sqrt\frac{1+3\delta}{1-\delta}\norm{X_0X_0^T-\tilde M^*}_F+2D\right).
\]
\end{theorem}

\section{Global Convergence}\label{sec:global}

Having developed local convergence results, the next step is to design an algorithm whose trajectory will eventually enter the local convergence region from any initial point. The major challenge is to deal with the saddle points outside the local regularity region. One common approach is the perturbed gradient descent method, which adds random noise to jump out of a neighborhood of a strict saddle point. Using the symmetric problem as an example, the basic idea is to first use the analysis in \citet{JGNK2017} to show that the perturbed gradient descent method will successfully find a matrix $X$ that approximately satisfies the first-order and second-order necessary optimality conditions, i.e.,
\begin{equation}\label{eq:approxoptimalitycond}
\norm{\nabla g_s(X)}_F \leq \kappa, \quad \lambda_{\min}(\nabla^2g_s(X)) \geq -\kappa,
\end{equation}
after a certain number of iterations where the number depends on $\kappa$. Here, $\lambda_{\min}(\nabla^2g_s(X))$ denotes the minimum eigenvalue of the matrix $\mathbf G$ that satisfies the equation
\[
(\vect(U))^T\mathbf G\vect(V)=[\nabla^2g_s(X)](U,V),
\]
for all $U,V \in \RR^{n \times r}$. The second step is to prove the strict saddle property for the problem, which means that for appropriate values of $\kappa$ the two conditions in \eqref{eq:approxoptimalitycond} imply that $\norm{XX^T-M^*}_F$ is so small that $X$ is in the local convergence region given by Theorem~\ref{thm:localconvsym}. After this iteration, the algorithm switches to the simple gradient descent method. This two-phase algorithm is commonly called the perturbed gradient descent method with local improvement \citep{JGNK2017}, whose details are given by Algorithm~\ref{alg:perturbed} in Appendix~\ref{app:globalproof}. The proofs in this section are also given in Appendix~\ref{app:globalproof}.

In this section, we will present two conditions that guarantee the global linear convergence of the above algorithm. For symmetric problems, the next lemma provides the strict saddle property and fulfills the purpose for the second step mentioned above. Its proof is a generalization of the one for the absence of spurious local minima under the same assumption in \citet{Zhang2021-p}.

\begin{lemma}\label{lem:strictsaddlelin}
Consider the symmetric problem \eqref{eq:symrec} with $\delta<1/2$. For every $C>0$, there exists some $\kappa>0$ such that for every $X \in \RR^{n \times r}$ the two conditions given in \eqref{eq:approxoptimalitycond} will imply $\norm{XX^T-ZZ^T}_F<C$.
\end{lemma}

The remaining step is to show that the trajectory of the perturbed gradient descent method will always belong to a bounded region in which the gradient and Hessian of the objective $g_s$ are Lipschitz continuous (see Appendix~\ref{app:globalproof}). Combining the above results with Theorem~3 in \citet{JGNK2017}, we can obtain the following global linear convergence result.

\begin{theorem}\label{thm:globalconvlin}
Consider the symmetric problem \eqref{eq:symrec} with $\delta<1/2$. For every $\epsilon>0$, the perturbed gradient descent method with local improvement under a suitable step size $\eta$ and perturbation size $w$ finds a solution $\hat X$ satisfying $\norm{\hat X\hat X^T-M^*}_F \leq \epsilon$ with high probability in $O(\log(1/\epsilon))$ number of iterations. Here, $\eta$ and $w$ are defined in Algorithm~\ref{alg:perturbed} in Appendix~\ref{app:globalproof}.
\end{theorem}

In the above theorem, the order $O(\log(1/\epsilon))$ of the convergence rate is determined by the number of iterations spent in the second phase of the algorithm, because the number of iterations in the first phase is independent of $\epsilon$. Note that we only show the relationship between the number of iterations and $\epsilon$, but the convergence rate also depends on the initial point $X_0$ and the loss function $f_s$. Moreover, although not being related to the final convergence rate, Theorem~3 in \citet{JGNK2017} also shows that the number of iterations in the first phase is polynomial with respect to the problem size.

For asymmetric problems with arbitrary objectives and rank $r$, if we apply the transformation from asymmetric problems to symmetric problems and replace $\delta$ in Theorem~\ref{thm:globalconvlin} with $2\delta/(1+\delta)$, Theorem~\ref{thm:globalconvlin} immediately implies the following global linear convergence result.

\begin{theorem}\label{thm:globalconv}
Consider the asymmetric problem \eqref{eq:asymrec} with $\delta<1/3$. For every $\epsilon>0$, the perturbed gradient descent method with local improvement under a suitable step size $\eta$ and perturbation size $w$ finds a solution $\hat X$ satisfying $\norm{\hat X\hat X^T-\tilde M^*}_F \leq \epsilon$ with high probability in $O(\log(1/\epsilon))$ number of iterations.
\end{theorem}

\section{Numerical Illustration}\label{sec:numerical}

\begin{figure*}[t]
\centering
\subfloat[]{\includegraphics[width=6cm]{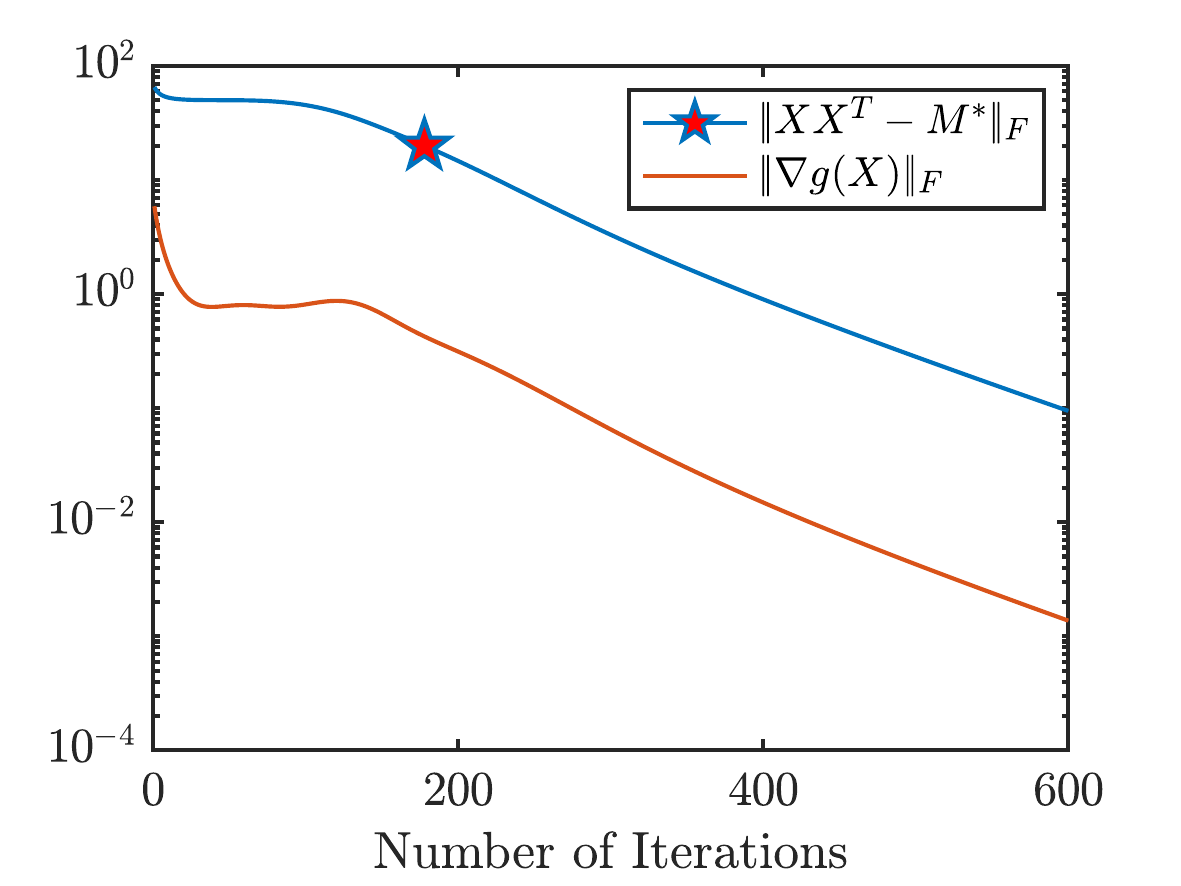}}\qquad
\subfloat[]{\includegraphics[width=6cm]{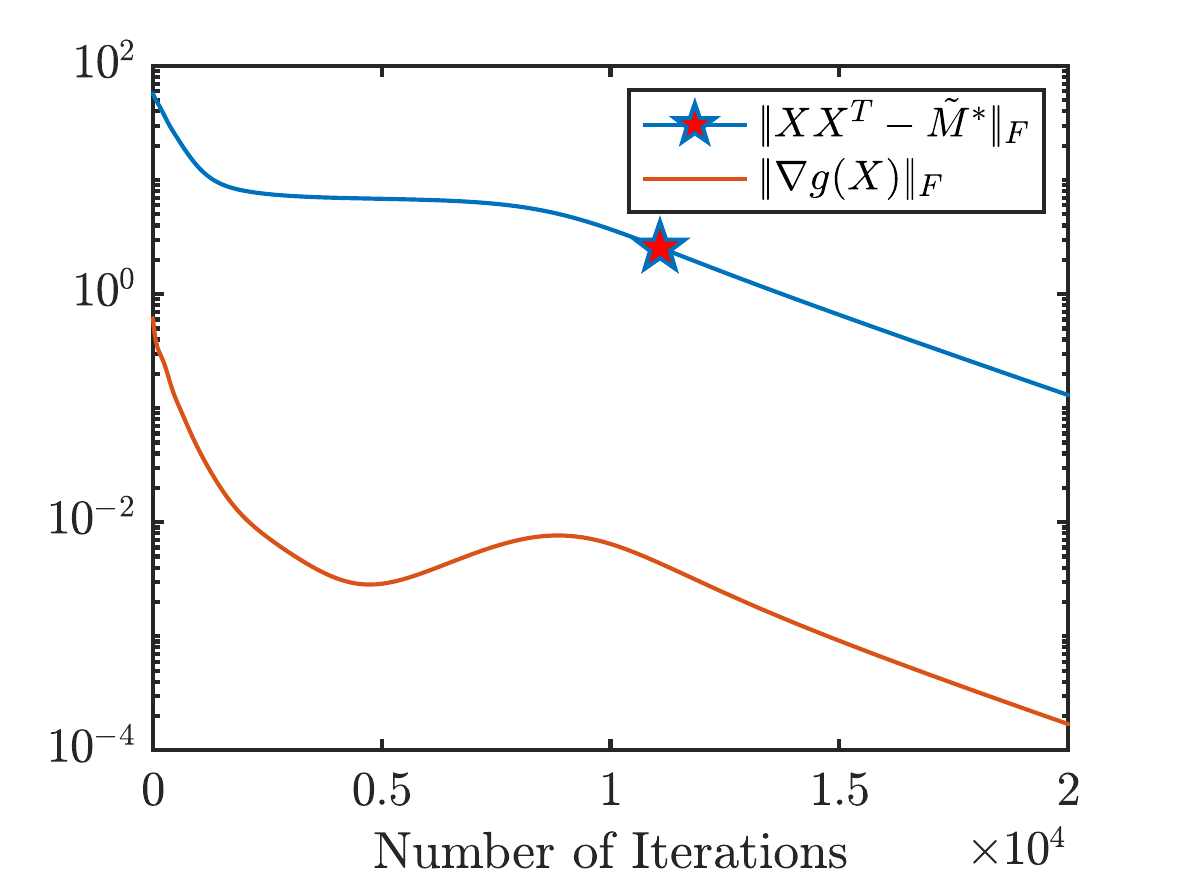}} \\
\subfloat[]{\includegraphics[width=6cm]{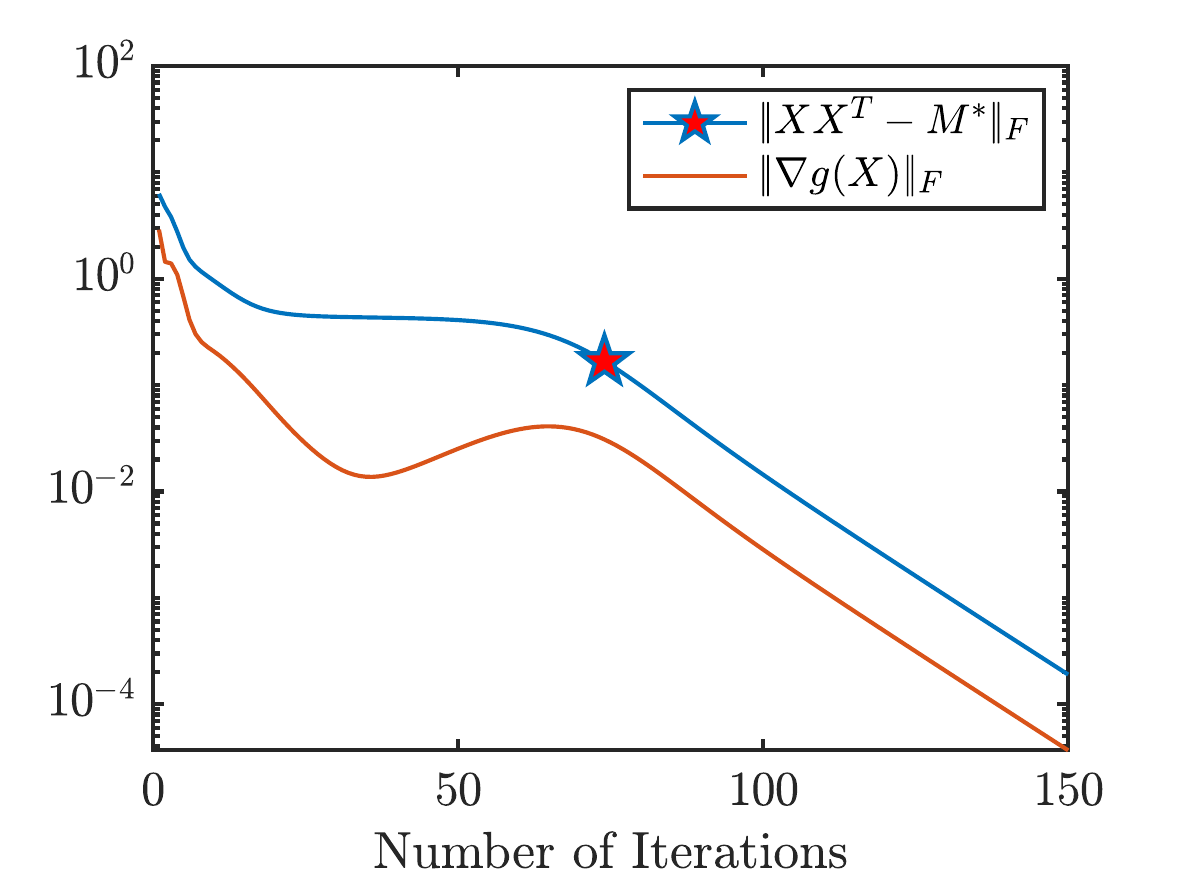}}\qquad
\subfloat[]{\includegraphics[width=6cm]{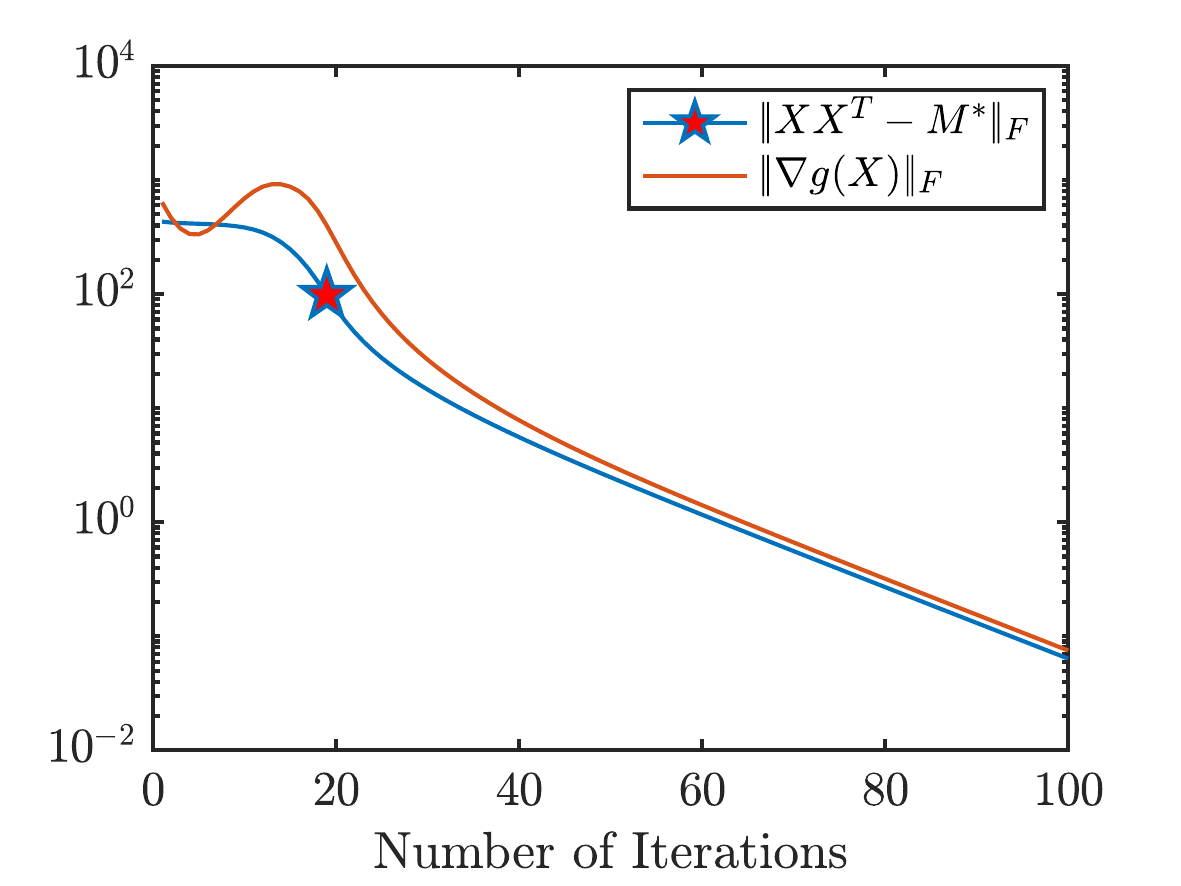}}
\caption{The trajectory of the perturbed gradient descent method for solving the low-rank matrix recovery problem. The marker in each figure shows the boundary of the local convergence region provided by Theorem~\ref{thm:localconvsym}. (a) A symmetric linear problem with $r=1$, $n=40$, $p=120$ and $\delta$ estimated to be $0.49$. (b) An asymmetric linear problem with $r=5$, $n=10$, $m=8$, $p=220$ and $\delta$ estimated to be $0.32$. (c) The 1-bit matrix recovery problem with $r=5$, $n=10$. (d) The 1-bit matrix recovery problem with $r=2$, $n=600$.}\label{fig:numerical}
\end{figure*}

In this section, we conduct numerical experiments to demonstrate the behavior of the perturbed gradient descent algorithm for solving low-rank matrix recovery problems. The linear convergence rate observed for the examples below supports our theoretical analyses in Section~\ref{sec:local} and Section~\ref{sec:global}.

In the first experiment, we consider the loss function \eqref{eq:linearobj} induced by a linear operator $\mathcal A$ with
\[
\mathcal A(M)=(\langle A_1,M\rangle,\dots,\langle A_p,M\rangle).
\]
Here, each entry of $A_i$ is independently generated from the standard Gaussian distribution. As shown in \citet{CP2011}, such linear operator $\mathcal A$ satisfies RIP with high probability if the number of measurements is large enough. Since it is NP-hard to check whether the resulting loss function $f_s$ or $f_a$ satisfies the $\delta$-$\RIP_{2r}$ for certain $\delta$, the $\delta$ parameter is estimated as follows: For the symmetric problem \eqref{eq:symrec}, we first generate $10^4$ random matrices $X \in \RR^{n \times 2r}$ with each entry independently selected from the standard Gaussian distribution, and then find the proper scaling factor $a \in \RR$ and the smallest $\delta$ such that
\[
(1-\delta)\norm{XX^T}_F^2 \leq \norm{a\mathcal A(XX^T)}^2 \leq (1+\delta)\norm{XX^T}_F^2
\]
holds for all generated matrices $X$. The $\delta$ parameter for the asymmetric problem \eqref{eq:asymrec} can be estimated similarly. After that, the ground truth $M^*=XX^T$ or $M^*=UV^T$ is generated randomly with each entry of $X$ or $(U,V)$ independently selected from the standard Gaussian distribution. The initial point is generated in the same way.

Figure~\ref{fig:numerical}(a) and (b) show the difference between the obtained solution and the ground truth together with the norm of the gradient of the objective function at different iterations. The convergence behavior clearly divides into two stages. The convergence rate is sublinear initially and then switches to linear when the current point moves into the local region associated with the PL inequality. In Figure~\ref{fig:numerical}(a) and (b), the marker shows the first time when the current point falls into the local convergence region provide in Theorem~\ref{thm:localconvsym} or Theorem~\ref{thm:localconvasym}. It can be seen that these theorems predict the boundary of the transition from a sublinear convergence rate to the linear convergence rate fairly tightly. After this point, $O(\log(1/\epsilon))$ additional iterations are needed to find an approximate solution with accuracy $\epsilon$. On the other hand, the occasion when perturbation needs to be added is rare in practice since it is unlikely for the trajectory to be very close to a saddle point. However, such perturbation is necessary theoretically to deal with pathological cases.

Second, we consider the 1-bit matrix recovery \citep{DPBW2014} with full measurements, which is a nonlinear low-rank matrix recovery problem. In this problem, there is an unknown symmetric ground truth matrix $\hat M \in \RR^{n \times n}$ with $\hat M \succeq 0$ and $\rank(\hat M)=r$. One is allowed to take independent measurements on every entry $\hat M_{ij}$, where each measurement value is a binary random variable whose distribution is given by $Y_{ij}=1$ with probability $\sigma(\hat M_{ij})$ and $Y_{ij}=0$ otherwise. Here, $\sigma(x)$ is commonly chosen to be the sigmoid function $\ee^x/(\ee^x+1)$. After a number of measurements are taken, let $y_{ij}$ be the percentage of the measurements on the $(i,j)$-th entry that are equal to $1$. The goal is to find the maximum likelihood estimator for the ground truth $\hat M$, which can be formulated as finding the global minimizer $M^*$ of the problem \eqref{eq:symrec} with
\[
f_s(M)=-\sum_{i=1}^n\sum_{j=1}^n(y_{ij}M_{ij}-\log(1+\ee^{M_{ij}})).
\]
To establish the RIP condition for the function $f_s$ above, consider its Hessian $\nabla^2f_s(M)$ that is given by
\[
[\nabla^2f_s(M)](K,L)=\sum_{i=1}^n\sum_{j=1}^n\sigma'(M_{ij})K_{ij}L_{ij},
\]
for every $M,K,L \in \RR^{n \times n}$. On the region
\begin{equation}\label{eq:onebitregion}
\{M \in \RR^{n \times n} | \, \abs{M_{ij}} \leq 2.29, \: \forall i,j=1,\dots,n\},
\end{equation}
we have $1/12<\sigma'(M_{ij}) \leq 1/4$, and thus the function $f_s$ satisfies the $\delta$-$\RIP_{2r}$ property with $\delta<1/2$.

Note that due to the noisy measurements the global minimizer $M^*$ is not equal to $\hat M$ in general. However, for demonstration purposes we should know $M^*$ a priori, and hence we consider the case when the number of measurements is large enough such that $y_{ij}=\sigma(\hat M_{ij})$ and $M^*=\hat M$. In Figure~\ref{fig:numerical}(c), the ground truth and the initial point are generated randomly in the region \eqref{eq:onebitregion}. Here, we can observe a similar two-stage convergence behavior as in the example with linear measurements. We experiment on the same problem with a larger matrix size $n$ as shown in Figure~\ref{fig:numerical}(d), which also gives similar results.

\section{Conclusion}

In this paper, we study the local and global convergence behaviors of gradient-based local search methods for solving low-rank matrix recovery problems in both symmetric and asymmetric cases. First, we present a novel method to identify a local region in which the PL inequality is satisfied, which is significantly larger than the region associated with the regularity conditions proven in the prior literature. This leads to a linear convergence result for the gradient descent method over a large local region. Second, we develop the strict saddle property for symmetric problems under the $\delta$-$\RIP_{2r}$ property with $\delta<1/2$. Then, we prove the global linear convergence of the perturbed gradient descent method for symmetric problems under the $\delta$-$\RIP_{2r}$ property with $\delta<1/2$, and the same convergence property can also be guaranteed for asymmetric problems with $\delta<1/3$. Compared with the existing results, these conditions are remarkably weaker and can be applied to a larger class of problems.

\bibliography{main-05122021}

\section*{Acknowledgments}

This work was supported by grants from AFOSR, ARO, ONR, and NSF.

\newpage
\onecolumn

\appendix

\section{Properties of the Factored Objectives}\label{app:factor}

We first study the smoothness properties for the gradient and Hessian of the objective function $g_s$ in the symmetric problem \eqref{eq:symrec}. The following lemma is borrowed from the proof of Theorem~7 in \citet{BL2020-pb}.

\begin{lemma}\label{lem:rop2r}
If $\mathcal Q$ is a quadratic form satisfying $\delta$-$\RIP_{2r}$, then
\[
\abs{[\mathcal Q](K,L)-\langle K,L\rangle} \leq 2\delta\norm{K}_F\norm{L}_F,
\]
for all matrices $K,L \in \RR^{n \times n}$ of rank at most $2r$.
\end{lemma}

\begin{lemma}\label{lem:smoothness}
For a given constant $R$ greater than $D$, the gradient $\nabla g_s$ of the function $g_s$ in the symmetric problem \eqref{eq:symrec} is $8\rho_1 r^{1/2}R$-restricted Lipschitz continuous and the Hessian $\nabla^2g_s$ is $4\rho_1r^{1/4}R^{1/2}(2r^{1/2}R\rho_2/\rho_1+3)$-restricted Lipschitz continuous over the region
\[
\mathcal D=\{X \in \RR^{n \times r} | \norm{XX^T}_F \leq R\}.
\]
\end{lemma}
\begin{proof}
For every $U \in \mathcal D$, we have
\begin{equation}\label{eq:normofproduct}
\norm{U}_F=\sqrt{\sum_{i=1}^r\sigma_i(U)^2} \leq \sqrt[4]{r\sum_{i=1}^r\sigma_i(U)^4}=\sqrt[4]{r\sum_{i=1}^r\lambda_i(UU^T)^2}=r^{1/4}\norm{UU^T}_F^{1/2} \leq r^{1/4}R^{1/2}.
\end{equation}
Furthermore, for every $U,V \in \mathcal D$, it holds that
\[
\norm{UU^T-VV^T}_F=\norm{U(U-V)^T+(U-V)V^T}_F \leq 2r^{1/4}R^{1/2}\norm{U-V}_F.
\]
To prove that the gradient $\nabla g_s$ is Lipschitz continuous, one can write
\begin{align*}
\norm{\nabla g_s(U)-\nabla g_s(V)}_F&=2\norm{\nabla f_s(UU^T)U-\nabla f_s(VV^T)V}_F \\
&\leq 2\norm{\nabla f_s(UU^T)U-\nabla f_s(VV^T)U}_F+2\norm{\nabla f_s(VV^T)(U-V)}_F \\
&\leq 2\rho_1\norm{UU^T-VV^T}_F\norm{U}_F+2\rho_1\norm{VV^T-M^*}_F\norm{U-V}_F \\
&\leq 4\rho_1 r^{1/2}R\norm{U-V}_F+4\rho_1 R\norm{U-V}_F \\
&\leq 8\rho_1 r^{1/2}R\norm{U-V}_F.
\end{align*}
Similarly, for every $W \in \RR^{n \times r}$, we have
\begin{align*}
[\nabla^2g_s(U)](W,W)&-[\nabla^2g_s(V)](W,W) \\
&=[\nabla^2f_s(UU^T)](UW^T+WU^T,UW^T+WU^T) \\
&\hspace{1em}-[\nabla^2f_s(VV^T)](VW^T+WV^T,VW^T+WV^T) \\
&\hspace{1em}+2\langle\nabla f_s(UU^T)-\nabla f_s(VV^T),WW^T\rangle \\
&=[\nabla^2f_s(UU^T)-\nabla^2f_s(VV^T)](UW^T+WU^T,UW^T+WU^T) \\
&\hspace{1em}+[\nabla^2f_s(VV^T)](UW^T+WU^T,UW^T+WU^T) \\
&\hspace{1em}-[\nabla^2f_s(VV^T)](VW^T+WV^T,VW^T+WV^T) \\
&\hspace{1em}+2\langle\nabla f_s(UU^T)-\nabla f_s(VV^T),WW^T\rangle.
\end{align*}
There are four terms in the above expression. The first term can be upper bounded as
\begin{align*}
\mathcal A_1&:=[\nabla^2f_s(UU^T)-\nabla^2f_s(VV^T)](UW^T+WU^T,UW^T+WU^T) \\
&\leq \rho_2\norm{UU^T-VV^T}_F\norm{UW^T+WU^T}_F^2 \\
&\leq 4\rho_2\norm{UU^T-VV^T}_F\norm{U}_F^2\norm{W}_F^2 \\
&\leq 8\rho_2r^{3/4}R^{3/2}\norm{U-V}_F\norm{W}_F^2.
\end{align*}
Similarly, the sum of the second and third terms can be bounded as
\begin{align*}
\mathcal A_2&:=[\nabla^2f_s(VV^T)](UW^T+WU^T,UW^T+WU^T) \\
&\hspace{1em}-[\nabla^2f_s(VV^T)](VW^T+WV^T,VW^T+WV^T) \\
&=[\nabla^2f_s(VV^T)](UW^T+WU^T,(U-V)W^T+W(U-V)^T) \\
&\hspace{1em}+[\nabla^2f_s(VV^T)]((U-V)W^T+W(U-V)^T,VW^T+WV^T) \\
&\leq (1+2\delta)(\norm{UW^T+WU^T}_F+\norm{VW^T+WV^T}_F)\norm{(U-V)W^T+W(U-V)^T}_F \\
&\leq 4(1+2\delta)(\norm{U}_F+\norm{V}_F)\norm{U-V}_F\norm{W}_F^2 \\
&\leq 8\rho_1 r^{1/4}R^{1/2}\norm{U-V}_F\norm{W}_F^2,
\end{align*}
where Lemma~\ref{lem:rop2r} is applied in the second step. Moreover, we can upper bound the last term as
\begin{align*}
\mathcal A_3&:=2\langle\nabla f_s(UU^T)-\nabla f_s(VV^T),WW^T\rangle \\
&\leq 2\rho_1\norm{UU^T-VV^T}_F\norm{W}_F^2 \\
&\leq 4\rho_1 r^{1/4}R^{1/2}\norm{U-V}_F\norm{W}_F^2.
\end{align*}
Therefore,
\begin{align*}
[\nabla^2g_s(U)](W,W)-[\nabla^2g_s(V)](W,W)&=\mathcal A_1+\mathcal A_2+\mathcal A_3 \\
&\leq 4\rho_1 r^{1/4}R^{1/2}(2r^{1/2}R\rho_2/\rho_1+3)\norm{U-V}_F\norm{W}_F^2,
\end{align*}
which implies that the Hessian $\nabla^2g_s$ has the desired Lipschitz property.
\end{proof}

Next, we verify some facts about the augmented ground truth $\tilde M^*$ for the asymmetric problem, which will be useful in the transformation from asymmetric problems to symmetric problems.

\begin{lemma}
The augmented ground truth $\tilde M^*$ defined in \eqref{eq:auggroundtruth} is independent of the balanced factorization of the ground truth $M^*$. Furthermore,
\[
\norm{\tilde M^*}_F=2\norm{M^*}_F, \quad \sigma_r(\tilde M^*)=2\sigma_r(M^*).
\]
\end{lemma}
\begin{proof}
By expanding all the terms, it can be checked that the inequality
\[
\norm{U_1U_1^T-U_2U_2^T}_F^2+\norm{V_1V_1^T-V_2V_2^T}_F^2 \leq 2\norm{U_1V_1^T-U_2V_2^T}_F^2
\]
holds for all $U_1,U_2 \in \RR^{n \times r}$ and $V_1,V_2 \in \RR^{m \times r}$ with $U_1^TU_1=V_1^TV_1$ and $U_2^TU_2=V_2^TV_2$ (see Appendix~F in \citet{ZLTW2018}). Then, if $(U_1,V_1)$ and $(U_2,V_2)$ are two balanced factorizations of the ground truth $M^*$, we must have
\[
U_1U_1^T=U_2U_2^T, \quad V_1V_1^T=V_2V_2^T
\]
and thus $\tilde M^*$ is unique.

Assume that $(U^*,V^*)$ is a balanced factorization of $M^*$, the remaining equalities follow from the fact that
\begin{align*}
\sigma_i(M^*)^2&=\sigma_i(U^*V^{*T}V^*U^{*T})=\sigma_i(U^*U^{*T}U^*U^{*T}) \\
&=\sigma_i(U^*U^{*T})^2=\sigma_i(U^{*T}U^*)^2 \\
&=\frac{1}{4}\sigma_i(U^{*T}U^*+V^{*T}V^*)^2=\frac{1}{4}\sigma_i(\tilde M^*)^2
\end{align*}
for all $i \in \{1,\dots,r\}$.
\end{proof}

In the following, we will show that the gradient and the Hessian of the objective function $g_a$ in the transformed asymmetric problem \eqref{eq:asymtrans} satisfies the same Lipschitz property as in Lemma~\ref{lem:smoothness}. This means that those proofs in the remainder of this paper that depend on the Lipschitz property of $g_s$ can be applied to both the symmetric problem \eqref{eq:symrec} and the transformed asymmetric problem \eqref{eq:asymtrans}.

\begin{lemma}\label{lem:smoothnessasym}
The gradient $\nabla g_a$ and the Hessian $\nabla^2g_a$ in the transformed asymmetric problem \eqref{eq:asymtrans} satisfy the same Lipschitz property as in Lemma~\ref{lem:smoothness}.
\end{lemma}
\begin{proof}
Consider arbitrary low-rank matrices $N,N',K \in \RR^{(n+m) \times (n+m)}$ written in block forms in the same way as in \eqref{eq:asymtransfunc}, with $\rank(N) \leq r$, $\rank(N') \leq r$ and $\rank(K) \leq 2r$. First, we will prove that the gradient $\nabla F$ and the Hessian $\nabla^2F$ of the transformed function $F$ are still $\rho_1$-restricted Lipschitz continuous and $\rho_2$-restricted Lipschitz continuous, respectively. Given the gradient
\[
\nabla F(N)=\frac{1}{2}\begin{bmatrix}
0 & \nabla f_a(N_{12}) \\
(\nabla f_a(N_{21}^T))^T & 0
\end{bmatrix}+\frac{\phi}{2}\begin{bmatrix}
N_{11} & -N_{12} \\
-N_{21} & N_{22}
\end{bmatrix},
\]
we have
\begin{align*}
\norm{\nabla F(N)&-\nabla F(N')}_F \\
&\leq \frac{1}{2}\sqrt{\norm{\nabla f_a(N_{12})-\nabla f_a(N_{12}')}_F^2+\norm{\nabla f_a(N_{21}^T)-\nabla f_a(N_{21}'^T)}_F^2}+\frac{\phi}{2}\norm{N-N'}_F \\
&\leq \frac{\rho_1}{2}\sqrt{\norm{N_{12}-N'_{12}}_F^2+\norm{N_{21}-N'_{21}}_F^2}+\frac{\phi}{2}\norm{N-N'}_F \\
&\leq \frac{1}{2}(\rho_1+\phi)\norm{N-N'}_F \leq \rho_1\norm{N-N'}_F,
\end{align*}
in which the second inequality is due to the $\rho_1$-restricted Lipschitz continuity of $\nabla f_a$, while the last inequality follows from the choice $\phi=(1-\delta)/2$ in Assumption~\ref{ass:regcoeff} and $\rho_1 \geq 1+2\delta$ in Assumption~\ref{ass:rip}. Moreover, since
\begin{align*}
[\nabla^2F(N)](K,K)&=\frac{1}{2}([\nabla^2f_a(N_{12})](K_{12},K_{12})+[\nabla^2f_a(N_{21}^T)](K_{21}^T,K_{21}^T)) \\
&+\frac{\phi}{2}(\norm{K_{11}}_F^2+\norm{K_{22}}_F^2-\norm{K_{12}}_F^2-\norm{K_{21}}_F^2),
\end{align*}
it is clear that $\nabla^2F$ is $\rho_2/2$-restricted Lipschitz continuous as the second term in the above equation is independent of $N$. Next, we can repeat the argument in Lemma~\ref{lem:smoothness} with the function $f_s$ replaced with $4F/(1+\delta)$, noting that the latter function satisfies the $2\delta/(1+\delta)$-$\RIP_{2r}$ property as proven in Theorem~12 of \citet{ZBL2021-p}.
\end{proof}

Using the Lipschitz properties proven in Lemma~\ref{lem:smoothness}, we will show that the objective value decreases at each iteration of the gradient descent algorithm with a sufficiently small step size $\eta$. Although the following lemma is stated for the symmetric problem \eqref{eq:symrec}, a similar result holds for the transformed asymmetric problem \eqref{eq:asymtrans}.

\begin{lemma}\label{lem:onestep}
Given a matrix $X \in \RR^{n \times r}$ satisfying
\[
\norm{XX^T-M^*}_F \leq R,
\]
let $X'=X-\eta\nabla g_s(X)$ be the result of a one-step gradient descent applied to the symmetric problem \eqref{eq:symrec} with the step size $\eta$ satisfying
\[
1/\eta \geq 12\rho_1 r^{1/2}(R+D)
\]
Then, $g_s(X') \leq g_s(X)-\eta\norm{\nabla g_s(X)}_F^2/2$.
\end{lemma}
\begin{proof}
The assumption on $\eta$ implies that $\eta\rho_1 R \leq 1/12$. Define $\tilde X(t)=X-t\eta\nabla g_s(X)$ for $t \in [0,1]$. We have $\tilde X(1)=X'$ and
\begin{align*}
\norm{\tilde X(t)&\tilde X(t)^T-M^*}_F \leq 2t\eta\norm{\nabla g_s(X)X^T}_F+t^2\eta^2\norm{\nabla g_s(X)\nabla g_s(X)^T}_F+\norm{XX^T-M^*}_F \\
&\leq 4t\eta\norm{\nabla f_s(XX^T)XX^T}_F+4t^2\eta^2\norm{\nabla f_s(XX^T)XX^T\nabla f_s(XX^T)^T}+R \\
&\leq 4\eta\rho_1\norm{XX^T-M^*}_F\norm{XX^T}_F+4\eta^2\rho_1^2\norm{XX^T-M^*}_F^2\norm{XX^T}_F+R \\
&\leq 4\eta\rho_1 R(R+D)(1+\eta\rho_1 R)+R \leq \frac{3}{2}(R+D).
\end{align*}
By the Lipschitz property of the function $g_s$ proven in Lemma~\ref{lem:smoothness} and the assumption on $\eta$, we have
\[
\norm{\nabla g_s(\tilde X(t))-\nabla g_s(X)}_F \leq 12\rho_1 r^{1/2}(R+D)\norm{\tilde X(t)-X}_F \leq t\norm{\nabla g_s(X)}_F.
\]
Now, one can write
\begin{align*}
g_s(X')-g_s(X)&=\int_0^1\langle\nabla g_s(\tilde X(t)),X'-X\rangle\dd t \\
&=-\eta\norm{\nabla g_s(X)}_F^2+\eta\int_0^1\langle\nabla g_s(X)-\nabla g_s(\tilde X(t)),\nabla g_s(X)\rangle\dd t \\
&\leq -\eta\norm{\nabla g_s(X)}_F^2+\frac{\eta}{2}\norm{\nabla g_s(X)}_F^2.
\end{align*}
As a result, $g_s(X') \leq g_s(X)-\eta\norm{\nabla g_s(X)}_F^2/2$.
\end{proof}

\section{Proofs for Section~\ref{sec:local}}\label{app:localproof}

First, we need to introduce some notations that will be used throughout this section and next two sections. For every $X \in \RR^{n \times r}$, define
\[
\mathbf e:=\vect(XX^T-M^*)
\]
and let $\mathbf X \in \RR^{n^2 \times nr}$ be the matrix satisfying
\[
\mathbf X\vect(U)=\vect(XU^T+UX^T), \quad \forall U \in \RR^{n \times r}.
\]
The following lemma is the key to the analysis of optimality conditions for the spurious local minima of the symmetric problem \eqref{eq:symrec}, which will be used in both this and next sections.

\begin{lemma}\label{lem:gradhessian}
For every $X \in \RR^{n \times r}$, there exists a symmetric matrix $\mathbf H \in \RR^{n^2 \times n^2}$ satisfying the $\delta$-$\RIP_{2r}$ property such that
\begin{gather*}
\norm{\mathbf X^T\mathbf H\mathbf e} \leq \norm{\nabla g_s(X)}_F, \\
2I_r \otimes \mat_S(\mathbf H\mathbf e)+(1+\delta)\mathbf X^T\mathbf X \succeq \lambda_{\min}(\nabla^2g_s(X))I_{nr}.
\end{gather*}
\end{lemma}
\begin{proof}
For given matrix $N \in \RR^{n \times n}$, define an auxiliary function $h_N:\RR^{n \times n} \to \RR$ by letting
\[
h_N(M)=\langle\nabla f_s(M),N\rangle, \quad \forall M \in \RR^{n \times n}.
\]
The mean value theorem over the function $h_N$ implies that
\begin{equation}\label{eq:meanvalue}
\begin{aligned}
\langle\nabla f_s(XX^T),N\rangle&=h_N(XX^T)-h_N(M^*) \\
&=\int_0^1\langle\nabla h_N((1-t)XX^T+tM^*),XX^T-M^*\rangle\dd t \\
&=\int_0^1[\nabla^2f_s((1-t)XX^T+tM^*)](XX^T-M^*,N)\dd t \\
&=\mathbf e^T\mathbf H\vect(N),
\end{aligned}
\end{equation}
where $\mathbf H \in \RR^{n^2 \times n^2}$ is the symmetric matrix that is independent of $N$ and satisfies
\[
(\vect(K))^T\mathbf H\vect(L)=\int_0^1[\nabla^2f_s((1-t)XX^T+tM^*)](K,L)\dd t
\]
for all $K,L \in \RR^{n \times n}$. Moreover, since $\nabla^2f_s((1-t)XX^T+tM^*)$ satisfies the $\delta$-$\RIP_{2r}$ property for all $t \in [0,1]$, $\mathbf H$ also satisfies the $\delta$-$\RIP_{2r}$. Now, we will prove the desired inequalities after choosing $\mathbf H$ as above.

First, let $U \in \RR^{n \times r}$ be the matrix satisfying $\vect(U)=\mathbf X^T\mathbf H\mathbf e$ and $N=XU^T+UX^T$. Then, by the equation \eqref{eq:meanvalue},
\[
\norm{\mathbf X^T\mathbf H\mathbf e}^2=\mathbf e^T\mathbf H\mathbf X\vect(U)=\mathbf e^T\mathbf H\vect(N)=\langle\nabla f_s(XX^T),N\rangle=\langle\nabla g_s(X),U\rangle \leq \norm{\nabla g_s(X)}_F\norm{U}_F,
\]
which arrives at the first inequality to be proved. Next, for every $U \in \RR^{n \times r}$ with $\mathbf U=\vect(U)$, the equation \eqref{eq:meanvalue} with $N=UU^T$ gives
\[
\langle\nabla f_s(XX^T),UU^T\rangle=\mathbf e^T\mathbf H\vect(UU^T)=\frac{1}{2}\mathbf U^T\vect((W+W^T)U)=\mathbf U^T(I_r \otimes \mat_S(\mathbf H\mathbf e))\mathbf U,
\]
in which $W \in \RR^{n \times n}$ is the unique matrix satisfying $\vect(W)=\mathbf H\mathbf e$. Therefore,
\begin{align*}
\lambda_{\min}(\nabla^2g_s(X))\norm{\mathbf U}^2 \leq [\nabla^2g_s(X)](U,U)&=[\nabla^2f_s(XX^T)](XU^T+UX^T,XU^T+UX^T)+2\langle\nabla f_s(XX^T),UU^T\rangle \\
&\leq (1+\delta)\norm{XU^T+UX^T}_F^2+2\langle\nabla f_s(XX^T),UU^T\rangle \\
&=(1+\delta)\mathbf U^T\mathbf X^T\mathbf X\mathbf U+2\mathbf U^T(I_r \otimes \mat_S(\mathbf H\mathbf e))\mathbf U,
\end{align*}
in which the second inequality is due to the $\delta$-$\RIP_{2r}$ property of the function $f_s$. This leads to the second inequality to be proved.
\end{proof}

The following lemma borrowed from \citet{BNS2016} will also be useful.

\begin{lemma}\label{lem:normcompare}
Let $X,Z \in \RR^{n \times r}$ be two arbitrary matrices such that $X^TZ$ is symmetric and positive semidefinite. It holds that
\[
\sigma_r(ZZ^T)\norm{X-Z}_F^2 \leq \frac{1}{2(\sqrt 2-1)}\norm{XX^T-ZZ^T}_F^2.
\]
\end{lemma}

\begin{proof}[Proof of Lemma~\ref{lem:localpl}]
Define
\begin{equation}\label{eq:qdef}
q_1=\sqrt{1-\frac{\tilde C^2}{2(\sqrt 2-1)\sigma_r(M^*)}}, \quad q_2=\frac{\sqrt 2\mu'}{\sigma_r(M^*)^{1/2}-\tilde C}.
\end{equation}
The assumption \eqref{eq:plradius} on $\tilde C$ implies that $\delta<q_1$, and thus one can always find a sufficiently small $\mu'>0$ such that
\begin{equation}\label{eq:deltacontra}
\frac{1-\delta}{1+\delta}>\frac{1-q_1+q_2}{1+q_1}.
\end{equation}

We choose $\mu=\mu'^2/(1+\delta)$. Assume on the contrary that
\[
\frac{1}{2}\norm{\nabla g_s(X)}_F^2<\mu(g_s(X)-f_s(M^*))
\]
at a particular matrix $X$ in the region \eqref{eq:plneighbor}. Obviously, $XX^T \neq M^*$. It results from \eqref{eq:globalrip} that
\[
\frac{1}{2}\norm{\nabla g_s(X)}_F^2<\mu(f_s(XX^T)-f_s(M^*)) \leq \frac{\mu(1+\delta)}{2}\norm{XX^T-M^*}_F^2,
\]
and thus
\[
\norm{\nabla g_s(X)}_F \leq \mu'\norm{XX^T-M^*}_F.
\]
Therefore, if we define $\delta_f^*(X,\mu')$ to be the optimal value of the optimization problem
\begin{equation}\label{eq:deltaoptfoc}
\begin{aligned}
\min_{\delta,\mathbf H} \quad & \delta \\
\st \quad & \norm{\mathbf X^T\mathbf H\mathbf e} \leq \mu'\norm{\mathbf e}, \\
& \text{$\mathbf H$ is symmetric and satisfies $\delta$-$\RIP_{2r}$},
\end{aligned}
\end{equation}
then Lemma~\ref{lem:gradhessian} shows that $\delta_f^*(X,\mu') \leq \delta$. However, Lemma~\ref{lem:deltaoptfocbound} (to be stated next) shows that
\[
\frac{1-\delta}{1+\delta} \leq \frac{1-\delta_f^*(X,\mu')}{1+\delta_f^*(X,\mu')} \leq \frac{1-q_1+q_2}{1+q_1},
\]
which contradicts the inequality \eqref{eq:deltacontra}.
\end{proof}

\begin{lemma}\label{lem:deltaoptfocbound}
If $X \in \RR^{n \times r}$ is a matrix in the region \eqref{eq:plneighbor} such that $XX^T \neq M^*$, then the optimal value $\delta_f^*(X,\mu')$ of the optimization problem \eqref{eq:deltaoptfoc} satisfies
\[
\frac{1-\delta_f^*(X,\mu')}{1+\delta_f^*(X,\mu')} \leq \frac{1-q_1+q_2}{1+q_1},
\]
where $q_1$ and $q_2$ are defined in \eqref{eq:qdef}.
\end{lemma}
\begin{proof}
Let $Z \in \mathcal Z$ be a global minimizer such that $ZZ^T=M^*$. The fact that $X$ is in the region \eqref{eq:plneighbor} implies that $\norm{X-Z}_F \leq \tilde C$. Without loss of generality, it can be assumed that $X^TZ$ is symmetric and positive semidefinite. If this is not the case, then we use the singular value decomposition $X^TZ=PDQ^T$ in which $P,Q \in \RR^{n \times n}$ are orthogonal and $D \in \RR^{n \times n}$ is diagonal. By defining $R=QP^T$, the matrix $ZR$ becomes another global minimizer and
\[
X^T(ZR)=PDQ^TQP^T=PDP^T \succeq 0,
\]
implying that we can continue the following argument with $ZR$ instead of $Z$.

The optimal value of the problem \eqref{eq:deltaoptfoc} is equal to that of the problem
\begin{equation}\label{eq:deltaoptfoclmi}
\begin{aligned}
\min_{\delta,\mathbf H} \quad & \delta \\
\st \quad & \begin{bmatrix}
I_{nr} & \mathbf X^T\mathbf H\mathbf e \\
(\mathbf X^T\mathbf H\mathbf e)^T & \mu'^2\norm{\mathbf e}^2
\end{bmatrix} \succeq 0, \\
& (1-\delta)I_{n^2} \preceq \mathbf H \preceq (1+\delta)I_{n^2}.
\end{aligned}
\end{equation}
This can be proved by applying Lemma~\ref{lem:lmireform} with $a=\mu'\norm{\mathbf e}$ and a sufficiently large $b$ such that both the optimal solutions of \eqref{eq:deltaoptfoc} and \eqref{eq:deltaoptfoclmi} satisfy the second constraint in \eqref{eq:opt} and \eqref{eq:lmi}. Now, define $\eta_f^*(X,\mu')$ to be the optimal value of the following optimization problem:
\begin{equation}\label{eq:etaoptfoc}
\begin{aligned}
\max_{\eta,\mathbf H} \quad & \eta \\
\st \quad & \begin{bmatrix}
I_{nr} & \mathbf X^T\mathbf H\mathbf e \\
(\mathbf X^T\mathbf H\mathbf e)^T & \mu'^2\norm{\mathbf e}^2
\end{bmatrix} \succeq 0, \\
& \eta I_{n^2} \preceq \mathbf H \preceq I_{n^2}.
\end{aligned}
\end{equation}
Note that the first constraint in \eqref{eq:deltaoptfoclmi} and \eqref{eq:etaoptfoc} is actually equivalent to $\norm{\mathbf X^T\mathbf H\mathbf e} \leq \mu'\norm{\mathbf e}$. Given any feasible solution $(\delta,\mathbf H)$ to the problem \eqref{eq:deltaoptfoclmi},
\[
\left(\frac{1-\delta}{1+\delta},\frac{1}{1+\delta}\mathbf H\right)
\]
is a feasible solution to the above problem \eqref{eq:etaoptfoc}. Therefore,
\begin{equation}\label{eq:deltaeta}
\eta_f^*(X,\mu') \geq \frac{1-\delta_f^*(X,\mu')}{1+\delta_f^*(X,\mu')}.
\end{equation}
To prove the desired inequality, it is sufficient to upper bound $\eta_f^*(X,\mu')$ by finding a feasible solution to the dual problem of \eqref{eq:etaoptfoc} given below:
\begin{equation}\label{eq:etafocdual}
\begin{aligned}
\min_{U_1,U_2,G,\lambda,y} \quad & \tr(U_2)+\mu'^2\norm{\mathbf e}^2\lambda+\tr(G), \\
\st \quad & \tr(U_1)=1, \\
& (\mathbf Xy)\mathbf e^T+\mathbf e(\mathbf Xy)^T=U_1-U_2, \\
& \begin{bmatrix}
G & -y \\
-y^T & \lambda
\end{bmatrix} \succeq 0, \\
& U_1 \succeq 0, \quad U_2 \succeq 0.
\end{aligned}
\end{equation}
As shown in the first part of the proof of Lemma~19 in \citet{BL2020-pb}, there exists a nonzero vector $y \in \RR^{nr}$ such that
\begin{equation}\label{eq:ynormineq}
\norm{\mathbf Xy}^2 \geq 2\sigma_r(XX^T)\norm{y}^2
\end{equation}
and
\[
\norm{\mathbf e-\mathbf Xy} \leq \norm{X-Z}_F^2.
\]
By Lemma~\ref{lem:normcompare}, we have
\[
\frac{\norm{\mathbf e-\mathbf Xy}}{\norm{\mathbf e}} \leq \frac{\norm{X-Z}_F^2}{\norm{XX^T-M^*}_F} \leq \sqrt\frac{1}{2(\sqrt 2-1)\sigma_r(M^*)}\tilde C<1.
\]
If $\theta$ is the angle between $\mathbf e$ and $\mathbf Xy$, then the above inequality implies that $\theta<\pi/2$ and
\[
\sin\theta \leq \frac{\norm{\mathbf e-\mathbf Xy}}{\norm{\mathbf e}} \leq \sqrt\frac{1}{2(\sqrt 2-1)\sigma_r(M^*)}.
\]
Therefore,
\begin{equation}\label{eq:yangleineq}
\cos\theta \geq q_1.
\end{equation}
On the other hand, the Wielandt--Hoffman theorem implies that
\[
\abs{\sigma_r(XX^T)^{1/2}-\sigma_r(M^*)^{1/2}}=\abs{\sigma_r(X)-\sigma_r(Z)} \leq \norm{X-Z}_F \leq \tilde C.
\]
Combining the above inequality and \eqref{eq:ynormineq} gives
\begin{equation}\label{eq:ynormineq2}
\norm{y} \leq \frac{\norm{\mathbf Xy}}{\sqrt 2(\sigma_r(M^*)^{1/2}-\tilde C)}.
\end{equation}

Let
\[
M=(\mathbf Xy)\mathbf e^T+\mathbf e(\mathbf Xy)^T,
\]
with $y$ given above, and decompose $M$ as
\[
M=[M]_+-[M]_-
\]
such that $[M]_+ \succeq 0$ and $[M]_- \succeq 0$. By Lemma~14 in \citet{ZSL2019}, we have
\begin{align*}
\tr([M]_+)&=\norm{\mathbf e}\norm{\mathbf Xy}(1+\cos\theta), \\
\tr([M]_-)&=\norm{\mathbf e}\norm{\mathbf Xy}(1-\cos\theta).
\end{align*}
Again, $\theta$ is the angle between $\mathbf e$ and $\mathbf Xy$. Then,
\begin{gather*}
U_1^*=\frac{[M]_+}{\tr([M]_+)}, \quad U_2^*=\frac{[M]_-}{\tr([M]_+)}, \\
G^*=\frac{1}{\lambda^*}y^*y^{*T}, \quad \lambda^*=\frac{\norm{y^*}}{\mu'\norm{\mathbf e}} \quad y^*=\frac{y}{\tr([M]_+)}
\end{gather*}
form a feasible solution to the dual problem \eqref{eq:etafocdual} with the objective value
\[
\frac{\tr([M]_-)+2\mu'\norm{\mathbf e}\norm{y}}{\tr([M]_+)}=\frac{1-\cos\theta+2\mu'\norm{y}/\norm{\mathbf Xy}}{1+\cos\theta}.
\]
The inequalities \eqref{eq:yangleineq} and \eqref{eq:ynormineq2} imply that
\[
\eta_f^*(X,\eta) \leq \frac{1-q_1+q_2}{1+q_1}.
\]
The proof is completed by the above inequality and \eqref{eq:deltaeta}.
\end{proof}

\begin{proof}[Proof of Theorem~\ref{thm:localconvsym}]
Define
\[
\tilde C=\sqrt\frac{1+\delta}{2(\sqrt 2-1)\sigma_r(M^*)(1-\delta)}\norm{X_0X_0^T-M^*}_F.
\]
Then, it follows from Lemma~\ref{lem:localpl} that there exists a constant $\mu>0$ such that the PL inequality
\[
\frac{1}{2}\norm{\nabla g_s(X)}_F^2 \geq \mu(g_s(X)-f_s(M^*))
\]
is satisfied in the region
\[
\mathcal D=\{X \in \RR^{n \times r} | \dist(X,\mathcal Z) \leq \tilde C\}.
\]
By Lemma~\ref{lem:normcompare}, in order to prove that a matrix $X$ belongs to $\mathcal D$, it suffices to show that
\begin{equation}\label{eq:directdist}
\norm{XX^T-M^*}_F \leq \sqrt{2(\sqrt 2-1)}\sigma_r(M^*)^{1/2}\tilde C=\sqrt\frac{1+\delta}{1-\delta}\norm{X_0X_0^T-M^*}_F.
\end{equation}

Next, we prove by induction that $X_t$ satisfies \eqref{eq:directdist} and $g_s(X_t) \leq g_s(X_{t-1})$ at each step of the iteration. Obviously, \eqref{eq:directdist} holds for $X_0$. At step $t$, by Lemma~\ref{lem:onestep}, the induction assumption
\[
\norm{X_{t-1}X_{t-1}^T-M^*}_F \leq \sqrt\frac{1+\delta}{1-\delta}\norm{X_0X_0^T-M^*}_F
\]
and our choice of the step size $\eta$ imply that $g_s(X_t) \leq g_s(X_{t-1}) \leq \dots \leq g_s(X_0)$. Then, the inequality \eqref{eq:levelset} immediately implies that $X_t$ satisfies \eqref{eq:directdist}.

Finally, since $X_t$ is guaranteed to be contained in a region satisfying the PL inequality for all $t$, we can apply Theorem~1 in \citet{KNS2016} to obtain
\[
g_s(X_t)-f_s(M^*) \leq (1-\mu\eta)^t(g_s(X_0)-f_s(M^*)).
\]
Now, \eqref{eq:convrate} follows from the above inequality and \eqref{eq:globalrip}.
\end{proof}

After the transformation from asymmetric problems to symmetric problems, the proof of Theorem~\ref{thm:localconvasym} is similar to that of Theorem~\ref{thm:localconvsym}, and thus it is omitted here.

\begin{algorithm}
\caption{Perturbed Gradient Descent Method With Local Improvement}\label{alg:perturbed}
\begin{algorithmic}
\State $R \gets 3D(1+\delta)/(1-\delta)$
\State $\ell_1 \gets 8\rho_1 r^{1/2}R$, \quad $\ell_2 \gets 4\rho_1r^{1/4}R^{1/2}(2r^{1/2}R\rho_2/\rho_1+3)$
\State $\hat\epsilon \gets \min\{\kappa,\kappa^2/\ell_2\}$, \quad $\Delta \gets 2(1+\delta)D^2$
\State $\chi \gets 3\max\{\log((nr\ell_1\Delta)/(c\hat\epsilon^2\gamma)),4\}$, \quad $\eta \gets c/\ell_1$, \quad $w \gets \sqrt c\hat\epsilon/(\chi^2\ell_1)$
\State $g_{\text{thres}} \gets \sqrt c\hat\epsilon/\chi^2$, \quad $f_{\text{thres}} \gets c\sqrt{\hat\epsilon^3/\ell_2}/\chi^3$, \quad $t_{\text{thres}} \gets \chi\ell_1/(c^2\sqrt{\ell_2\hat\epsilon})$
\State $t \gets 0$, \quad $t_{\text{noise}} \gets -t_{\text{thres}}-1$
\Loop
\If{$\norm{\nabla g_s(X_t)}_F \leq g_{\text{thres}}$ and $t-t_{\text{noise}}>t_{\text{thres}}$}
\State $\tilde X_t \gets X_t$, \quad $t_{\text{noise}} \gets t$
\State $X_t \gets X_t+W$, where $W$ is drawn uniformly from the ball with radius $w$
\EndIf
\If{$t-t_{\text{noise}}=t_{\text{thres}}$ and $g_s(X_t)-g_s(\tilde X_{t_{\text{noise}}})>-f_{\text{thres}}$}
\State $X_t \gets \tilde X_{t_{\text{noise}}}$
\State\textbf{break}
\EndIf
\State $X_{t+1} \gets X_t-\eta\nabla g_s(X_t)$, \quad $t \gets t+1$
\EndLoop
\Loop
\State $X_{t+1} \gets X_t-\eta\nabla g_s(X_t)$, \quad $t \gets t+1$
\EndLoop
\end{algorithmic}
\end{algorithm}

\section{Proofs for Section~\ref{sec:global}}\label{app:globalproof}

We first present the perturbed gradient descent algorithm with local improvement adapted from the general algorithm in \citet{JGNK2017} for solving the symmetric problem \eqref{eq:symrec}, which can be also used to solve \eqref{eq:asymrec} after the transformation from asymmetric problems to symmetric problems. In Algorithm~\ref{alg:perturbed}, $X_0$ is the initial point and $1-\gamma$ is the success probability of the algorithm, while $\eta$ and $w$ are respectively the step size and perturbation size which are further determined by the parameter $c$. Furthermore, the parameter $\kappa$ determines at what time the corresponding point is sufficiently close to the ground truth so that it belongs to the local convergence region and thus perturbations are no longer necessary in future iterations. After the first loop ends, the current matrix $X_t$ will satisfy \eqref{eq:approxoptimalitycond}. The choice of the parameters $c$ and $\kappa$ will be given in the proof of Theorem~\ref{thm:globalconvlin}, but they can also be selected empirically.

The following lemma will be useful in the proof of Lemma~\ref{lem:strictsaddlelin}, which can be obtained by combining Lemma~6 and Lemma~7 in \citet{ZBL2021-p}.

\begin{lemma}\label{lem:strictsaddlesv}
For any $C>0$, there exist some $\kappa>0$ and $\zeta>0$ such that for each $X \in \RR^{n \times r}$ the two inequalities in \eqref{eq:approxoptimalitycond} together with $\sigma_r(X) \leq \zeta$ will imply $\norm{XX^T-ZZ^T}_F<C$.
\end{lemma}

\begin{proof}[Proof of Lemma~\ref{lem:strictsaddlelin}]
Let $\zeta$ be the constant given by Lemma~\ref{lem:strictsaddlesv}. We only need to consider all $X \in \RR^{n \times r}$ satisfying $\sigma_r(X)>\zeta$, since the opposite case can be directly handled by applying Lemma~\ref{lem:strictsaddlesv}. By Lemma~\ref{lem:gradhessian}, if $X$ satisfies the approximate first-order and second-order necessary optimality conditions \eqref{eq:approxoptimalitycond}, we must have $\delta \geq \delta^*(X,\kappa)$, where $\delta^*(X,\kappa)$ is the optimal value of the following optimization problem:
\begin{equation}\label{eq:deltaopt}
\begin{aligned}
\min_{\delta,\mathbf H} \quad & \delta \\
\st \quad & \norm{\mathbf X^T\mathbf H\mathbf e} \leq \kappa, \\
& 2I_r \otimes \mat_S(\mathbf H\mathbf e)+(1+\delta)\mathbf X^T\mathbf X \succeq -\kappa I_{nr}, \\
& \text{$\mathbf H$ is symmetric and satisfies $\delta$-$\RIP_{2r}$}.
\end{aligned}
\end{equation}
On the other hand, both the assumption $\delta<1/2$ and Lemma~\ref{lem:deltaoptbound} imply that
\[
\frac{1}{3}<\frac{1-\delta}{1+\delta} \leq \frac{1-\delta^*(X,\kappa)}{1+\delta^*(X,\kappa)} \leq \frac{1}{3}+\Gamma\frac{\kappa}{\norm{\mathbf e}},
\]
for some constant $\Gamma$, which further implies that
\[
\kappa \geq \frac{\norm{\mathbf e}}{\Gamma}\left(\frac{1-\delta}{1+\delta}-\frac{1}{3}\right).
\]
The strict saddle property can then be proved by choosing a sufficiently small $\kappa$.
\end{proof}

\begin{lemma}\label{lem:deltaoptbound}
Given a constant $\zeta>0$, if $X \in \RR^{n \times r}$ is a matrix satisfying $XX^T \neq M^*$ and $\sigma_r(X)>\zeta$, then the optimal value $\delta^*(X,\kappa)$ of the optimization problem \eqref{eq:deltaopt} satisfies
\[
\frac{1-\delta^*(X,\kappa)}{1+\delta^*(X,\kappa)} \leq \frac{1}{3}+\Gamma\frac{\kappa}{\norm{\mathbf e}},
\]
where $\Gamma=\sqrt r+\sqrt 2/\zeta$.
\end{lemma}
\begin{proof}
Let $Z \in \mathcal Z$ be a global minimizer such that $ZZ^T=M^*$. By Lemma~\ref{lem:lmireform} with $a=b=\kappa$ and an argument similar to the one in the proof of Lemma~\ref{lem:deltaoptfocbound}, we can introduce a relaxed optimization problem
\begin{equation}\label{eq:etaopt}
\begin{aligned}
\max_{\eta,\mathbf H} \quad & \eta \\
\st \quad & \begin{bmatrix}
I_{nr} & \mathbf X^T\mathbf H\mathbf e \\
(\mathbf X^T\mathbf H\mathbf e)^T & \kappa^2
\end{bmatrix} \succeq 0, \\
& 2I_r \otimes \mat_S(\mathbf H\mathbf e)+\mathbf X^T\mathbf X \succeq -\kappa I_{nr}, \\
& \eta I_{n^2} \preceq \mathbf H \preceq I_{n^2},
\end{aligned}
\end{equation}
whose optimal value $\eta^*(X,\kappa)$ satisfies
\[
\eta^*(X,\kappa) \geq \frac{1-\delta^*(X,\kappa)}{1+\delta^*(X,\kappa)}.
\]
To prove the desired inequality, we need to find an upper bound for $\eta^*(X,\kappa)$, which can be achieved by finding a feasible solution to the dual problem of \eqref{eq:etaopt}:
\begin{equation}\label{eq:etaoptdual}
\begin{aligned}
\min_{\substack{U_1,U_2,W, \\ G,\lambda,y}} \quad & \tr(U_2)+\langle\mathbf X^T\mathbf X,W\rangle+\kappa\tr(W)+\kappa^2\lambda+\tr(G) \\
\st \quad & \tr(U_1)=1, \\
& (\mathbf Xy-w)\mathbf e^T+\mathbf e(\mathbf Xy-w)^T=U_1-U_2, \\
& \begin{bmatrix}
G & -y \\
-y^T & \lambda
\end{bmatrix} \succeq 0, \\
& U_1 \succeq 0, \quad U_2 \succeq 0, \quad W=\begin{bmatrix}
W_{1,1} & \cdots & W_{r,1}^T \\
\vdots & \ddots & \vdots \\
W_{r,1} & \cdots & W_{r,r}
\end{bmatrix} \succeq 0, \\
& w=\sum_{i=1}^r\vect(W_{i,i}).
\end{aligned}
\end{equation}

Before describing the choice of the dual feasible solution, we need to represent the error vector $\mathbf e$ in a different form. Let $\mathcal P \in \RR^{n \times n}$ be the orthogonal projection matrix onto the range of $X$, and $\mathcal P_\perp \in \RR^{n \times n}$ be the orthogonal projection matrix onto the orthogonal complement of the range of $X$. Then, $Z$ can be decomposed as $Z=\mathcal PZ+\mathcal P_\perp Z$, and there exists a matrix $R \in \RR^{r \times r}$ such that $\mathcal PZ=XR$. Note that
\[
ZZ^T=\mathcal PZZ^T\mathcal P+\mathcal PZZ^T\mathcal P_\perp+\mathcal P_\perp ZZ^T\mathcal P+\mathcal P_\perp ZZ^T\mathcal P_\perp.
\]
Thus, if we choose
\begin{equation}\label{eq:ydef}
\hat Y=\frac{1}{2}X-\frac{1}{2}XRR^T-\mathcal P_\perp ZR^T, \quad \hat y=\vect(\hat Y),
\end{equation}
then it can be verified that
\begin{gather*}
X\hat Y^T+\hat YX^T-\mathcal P_\perp ZZ^T\mathcal P_\perp=XX^T-ZZ^T, \\
\langle X\hat Y^T+\hat YX^T,\mathcal P_\perp ZZ^T\mathcal P_\perp\rangle=0.
\end{gather*}
Moreover, we have
\begin{equation}\label{eq:ynorm}
\begin{aligned}
\norm{X\hat Y^T+\hat YX^T}_F^2&=2\tr(X^TX\hat Y^T\hat Y)+\tr(X^T\hat YX^T\hat Y)+\tr(\hat Y^TX\hat Y^TX) \\
&\geq 2\tr(X^TX\hat Y^T\hat Y) \geq 2\sigma_r(X)^2\norm{\hat Y}_F^2,
\end{aligned}
\end{equation}
in which the first inequality is due to
\[
\tr(X^T\hat YX^T\hat Y)=\frac{1}{4}\tr((X^TX(I_r-RR^T))^2)=\frac{1}{4}\tr((X(I_r-RR^T)X^T)^2) \geq 0.
\]

Assume first that $Z_\perp=\mathcal P_\perp Z \neq 0$. The other case will be handled at the end of this proof. In the case when $Z_\perp \neq 0$, we also have $X\hat Y^T+\hat YX^T \neq 0$. Otherwise, the inequality \eqref{eq:ynorm} and the assumption $\sigma_r(X)>0$ imply that $\hat Y=0$. The orthogonality and the definition of $\hat Y$ in \eqref{eq:ydef} then give
\[
X-XRR^T=0, \quad \mathcal P_\perp ZR^T=0.
\]
The first equation above implies that $R$ is invertible since $X$ has full column rank, which contradicts $Z_\perp \neq 0$. Now, define the unit vectors
\[
\hat u_1=\frac{\mathbf X\hat y}{\norm{\mathbf X\hat y}}, \quad \hat u_2=\frac{\vect(Z_\perp Z_\perp^T)}{\norm{Z_\perp Z_\perp^T}_F}.
\]
Then, $\hat u_1 \perp \hat u_2$ and
\begin{equation}\label{eq:epolar}
\mathbf e=\norm{\mathbf e}(\sqrt{1-\alpha^2}\hat u_1-\alpha\hat u_2)
\end{equation}
with
\begin{equation}\label{eq:alphadef}
\alpha=\frac{\norm{Z_\perp Z_\perp^T}_F}{\norm{XX^T-ZZ^T}_F}.
\end{equation}

We first describe our choices of the dual variables $W$ and $y$ (which will be scaled later). Let
\[
X^TX=QSQ^T, \quad Z_\perp Z_\perp^T=PGP^T,
\]
with $Q,P$ orthogonal and $S,G$ diagonal, such that $S_{11}=\sigma_r(X)^2$. Fix a constant $\gamma \in [0,1]$ that is to be determined and define
\begin{gather*}
V_i=k^{1/2}G_{ii}^{1/2}PE_{i1}Q^T, \quad \forall i=1,\dots,r, \\
W=\sum_{i=1}^r\vect(V_i)\vect(V_i)^T, \quad y=l\hat y,
\end{gather*}
with $\hat y$ defined in \eqref{eq:ydef} and
\[
k=\frac{\gamma}{\norm{\mathbf e}\norm{Z_\perp Z_\perp^T}_F}, \quad l=\frac{\sqrt{1-\gamma^2}}{\norm{\mathbf e}\norm{\mathbf X\hat y}}.
\]
Here, $E_{ij}$ is the elementary matrix of size $n \times r$ with the $(i,j)$-entry being $1$. By our construction, $X^TV_i=0$, which implies that
\begin{equation}\label{eq:dualobj1}
\langle\mathbf X^T\mathbf X,W\rangle=\sum_{i=1}^r\norm{XV_i^T+V_iX^T}_F^2=2\sum_{i=1}^r\tr(X^TXV_i^TV_i)=2k\sigma_r(X)^2\sum_{i=1}^rG_{ii}=2\beta\gamma,
\end{equation}
with
\begin{equation}\label{eq:betadef}
\beta=\frac{\sigma_r(X)^2\tr(Z_\perp Z_\perp^T)}{\norm{XX^T-ZZ^T}_F\norm{Z_\perp Z_\perp^T}_F}.
\end{equation}
In addition,
\begin{equation}\label{eq:dualobj2}
\tr(W)=\sum_{i=1}^r\norm{V_i}_F^2=k\sum_{i=1}^rG_{ii}=k\tr(Z_\perp Z_\perp^T) \leq \frac{\sqrt r}{\norm{\mathbf e}},
\end{equation}
and
\[
w=\sum_{i=1}^r\vect(W_{i,i})=\sum_{i=1}^rV_iV_i^T=kZ_\perp Z_\perp^T.
\]
Therefore,
\[
\mathbf Xy-w=\frac{1}{\norm{\mathbf e}}(\sqrt{1-\gamma^2}\hat u_1-\gamma\hat u_2),
\]
which together with \eqref{eq:epolar} implies that
\begin{equation}\label{eq:dualobj3}
\norm{\mathbf e}\norm{\mathbf Xy-w}=1, \quad \langle\mathbf e,\mathbf Xy-w\rangle=\gamma\alpha+\sqrt{1-\gamma^2}\sqrt{1-\alpha^2}=\psi(\gamma).
\end{equation}
Next, the inequality \eqref{eq:ynorm} and the assumption $\sigma_r(X)>\zeta$ imply that
\begin{equation}\label{eq:dualobj4}
\norm{y} \leq \frac{\sqrt{1-\gamma^2}}{\sqrt 2\zeta\norm{\mathbf e}} \leq \frac{1}{\sqrt 2\zeta\norm{\mathbf e}}.
\end{equation}

Define
\[
M=(\mathbf Xy-w)\mathbf e^T+\mathbf e(\mathbf Xy-w)^T
\]
and decompose
\[
M=[M]_+-[M]_-,
\]
in which both $[M]_+ \succeq 0$ and $[M]_- \succeq 0$. Let $\theta$ be the angle between $\mathbf e$ and $\mathbf Xy-w$. By Lemma~14 in \citet{ZSL2019}, we have
\begin{gather*}
\tr([M]_+)=\norm{\mathbf e}\norm{\mathbf Xy-w}(1+\cos\theta), \\
\tr([M]_-)=\norm{\mathbf e}\norm{\mathbf Xy-w}(1-\cos\theta).
\end{gather*}
Now, one can verify that
\begin{gather*}
U_1^*=\frac{[M]_+}{\tr([M]_+)}, \quad U_2^*=\frac{[M]_-}{\tr([M]_+)}, \\
y^*=\frac{y}{\tr([M]_+)}, \quad W^*=\frac{W}{\tr([M]_+)}, \\
\lambda^*=\frac{\norm{y^*}}{\kappa}, \quad G^*=\frac{1}{\lambda^*}y^*y^{*T}
\end{gather*}
forms a feasible solution to the dual problem \eqref{eq:etaoptdual} whose objective value is equal to
\[
\frac{\tr([M]_-)+\langle\mathbf X^T\mathbf X,W\rangle+\kappa\tr(W)+2\kappa\norm{y}}{\tr([M]_+)}.
\]
Putting \eqref{eq:dualobj1}, \eqref{eq:dualobj2}, \eqref{eq:dualobj3} and \eqref{eq:dualobj4} into the above equation, we can obtain
\[
\eta^*(X,\kappa) \leq \frac{2\beta\gamma+1-\psi(\gamma)+(\sqrt r+\sqrt 2/\zeta)\kappa/\norm{\mathbf e}}{1+\psi(\gamma)} \leq \frac{2\beta\gamma+1-\psi(\gamma)}{1+\psi(\gamma)}+\Gamma\frac{\kappa}{\norm{\mathbf e}}.
\]
Choosing the best $\gamma \in [0,1]$ to minimize the far right-side of the above inequality leads to
\[
\eta^*(X,\kappa) \leq \eta_0(X)+\Gamma\frac{\kappa}{\norm{\mathbf e}},
\]
with
\[
\eta_0(X)=\begin{dcases*}
\frac{1-\sqrt{1-\alpha^2}}{1+\sqrt{1-\alpha^2}}, & if $\beta \geq \dfrac{\alpha}{1+\sqrt{1-\alpha^2}}$, \\
\frac{\beta(\alpha-\beta)}{1-\beta\alpha}, & if $\beta \leq \dfrac{\alpha}{1+\sqrt{1-\alpha^2}}$.
\end{dcases*}
\]
Here, $\alpha$ and $\beta$ are defined in \eqref{eq:alphadef} and \eqref{eq:betadef}, respectively. In the proof of Theorem~1.2 in \citet{Zhang2021-p}, it is shown that $\eta_0(X) \leq 1/3$ for every $X$ with $XX^T \neq ZZ^T$, which gives our desired inequality.

Finally, we still need to deal with the case when $\mathcal P_\perp Z=0$. In this case, we know that $\mathbf X\hat y=\mathbf e$ with $\hat y$ defined in \eqref{eq:ydef}. Then, it is easy to check that
\begin{gather*}
U_1^*=\frac{\mathbf e\mathbf e^T}{\norm{\mathbf e}^2}, \quad U_2^*=0, \\
y^*=\frac{\hat y}{2\norm{\mathbf e}^2}, \quad W^*=0, \\
\lambda^*=\frac{\norm{y^*}}{\kappa}, \quad G^*=\frac{1}{\lambda^*}y^*y^{*T}
\end{gather*}
forms a feasible solution to the dual problem \eqref{eq:etaoptdual} whose objective value is $2\kappa\norm{y^*}$, which is at most $\kappa/(\sqrt 2\zeta\norm{\mathbf e})$ by the inequality \eqref{eq:ynorm}.
\end{proof}

\begin{lemma}\label{lem:lipschitzbound}
Consider Algorithm~\ref{alg:perturbed} for solving the symmetric problem \eqref{eq:symrec}. If the initial matrix $X_0$ satisfies
\[
\norm{X_0X_0^T}_F \leq D,
\]
the step size $\eta$ satisfies
\[
1/\eta \geq 48\rho_1 r^{1/2}\left(\frac{1+\delta}{1-\delta}D\right),
\]
and the perturbation size $w$ satisfies
\[
2wr^{1/4}\left(\frac{1+\delta}{1-\delta}\right)^{1/4}\sqrt{3D}+w^2 \leq \sqrt{\frac{1+\delta}{1-\delta}}D,
\]
then during the first loop the trajectory $X_t$ is always confined in the region
\[
\mathcal D=\left\{X \in \RR^{n \times r} \middle| \norm{XX^T-M^*}_F \leq 3\left(\frac{1+\delta}{1-\delta}\right)D\right\}.
\]
\end{lemma}
\begin{proof}
For convenience, we introduce the set
\[
\mathcal D_1=\left\{X \in \RR^{n \times r} \middle| \norm{XX^T-M^*}_F \leq 2\sqrt{\frac{1+\delta}{1-\delta}}D\right\}.
\]
The iteration is initialized at the point $X_0 \in \mathcal D_1$. Assume that at some time instance $t$ the current matrix $X_t \in \mathcal D_1$, $g_s(X_t) \leq g_s(X_0)$, and some perturbation needs to be added because $\norm{\nabla g_s(X_t)}_F$ is small. In this case, a random noise $W$ is generated from the uniform distribution in the ball of radius $w$. The algorithm saves the original point $X_t$ to $\tilde X_t$ and replaces $X_t$ with $X_t+W$. Then, similar to the inequality \eqref{eq:normofproduct}, the old point $\tilde X_t$ satisfies
\[
\norm{\tilde X_t}_F \leq r^{1/4}\left(\frac{1+\delta}{1-\delta}\right)^{1/4}\sqrt{3D},
\]
and thus the new point $X_t$ satisfies
\begin{align*}
\norm{X_tX_t^T-M^*}_F &\leq \norm{\tilde X_t\tilde X_t^T-M^*}_F+\norm{W\tilde X_t^T+X_t\tilde W^T}_F+\norm{WW^T}_F \\
&\leq 2\sqrt{\frac{1+\delta}{1-\delta}}D+2wr^{1/4}\left(\frac{1+\delta}{1-\delta}\right)^{1/4}\sqrt{3D}+w^2 \\
&\leq 3\sqrt{\frac{1+\delta}{1-\delta}}D,
\end{align*}
by our choice of the parameter $w$. Due to the design of the perturbed gradient descent algorithm, the perturbation will never be taken in the next $t_{\text{thres}}$ number of iterations ($t_{\text{thres}}$ is defined in Algorithm~\ref{alg:perturbed}). As a result, Lemma~\ref{lem:onestep}, $X_t \in \mathcal D$ and our choice of the step size $\eta$ imply that $g_s(X_{t+1}) \leq g_s(X_t)$. Hence, the inequality \eqref{eq:levelset} gives
\[
\norm{X_{t+1}X_{t+1}^T-M^*}_F \leq \sqrt{\frac{1+\delta}{1-\delta}}\norm{X_tX_t^T-M^*}_F \leq 3\left(\frac{1+\delta}{1-\delta}\right)D,
\]
which shows that $X_{t+2} \in \mathcal D$. Repeating this argument, it can be concluded that $g_s(X_{t+k}) \leq g_s(X_t)$ and $X_{t+k} \in \mathcal D$ for all $k=1,\dots,t_{\text{thres}}$. After $X_{t+t_{\text{thres}}}$ is obtained, the algorithm compares $g_s(X_{t+t_{\text{thres}}})$ with $g_s(\tilde X_t)$, and the iteration continues only if $g_s(X_{t+t_{\text{thres}}}) \leq g_s(\tilde X_t)$. When this is the case, $g_s(X_{t+t_{\text{thres}}}) \leq g_s(X_0)$, and by the inequality \eqref{eq:levelset} again, we have
\[
\norm{X_{t+t_{\text{thres}}}X_{t+t_{\text{thres}}}^T-M^*}_F \leq \sqrt{\frac{1+\delta}{1-\delta}}\norm{X_0X_0^T-M^*}_F \leq 2\sqrt{\frac{1+\delta}{1-\delta}}D,
\]
and thus $X_{t+t_{\text{thres}}} \in \mathcal D_1$. Assume that no perturbation is added at steps $t+t_{\text{thres}}+1,\dots,t+t_{\text{thres}}+l-1$. Then, using a similar argument as above, we can prove that
\[
g_s(X_{t+t_{\text{thres}}+k}) \leq g_s(X_{t+t_{\text{thres}}}) \leq g_s(X_0), \quad X_{t+t_{\text{thres}}+k} \in \mathcal D_1, \quad \forall k=1,\dots,l-1.
\]
If perturbation needs to be added at step $t+t_{\text{thres}}+l$, we can repeat the above argument with $t+t_{\text{thres}}+l$ instead of $t$, which leads to the desired result.
\end{proof}

\begin{proof}[Proof of Theorem~\ref{thm:globalconvlin}]
In the first stage of the algorithm, the perturbed gradient descent method is applied. If the parameter $c$ is sufficiently small, then the step size $\eta$ and the perturbation size $w$ will satisfy the assumptions in Lemma~\ref{lem:lipschitzbound}. In this case, Lemma~\ref{lem:lipschitzbound} implies that the iterations are taken within a region in which $\nabla g_s$ and $\nabla^2g_s$ are Lipschitz continuous. Let $\kappa$ be the constant given by Lemma~\ref{lem:strictsaddlelin} such that the approximate second-order necessary optimality conditions \eqref{eq:approxoptimalitycond} will imply that $\norm{XX^T-ZZ^T}_F<C$, where
\[
C=2(\sqrt 2-1)(1-\delta)\sigma_r(M^*)
\]
is the radius of the local linear convergence region provided by Theorem~\ref{thm:localconvsym}. Now, Theorem~3 in \citet{JGNK2017} shows that with probability $1-\gamma$ the first loop will stop with a solution $\tilde X$ satisfying \eqref{eq:approxoptimalitycond}, and thus $\tilde X$ is within the local convergence region. Note that the number of iterations in this stage is fixed for a given initial matrix $X_0$, and that this number is independent of $\epsilon$.

Next, the gradient descent algorithm is run with initialization at the matrix $\tilde X$. Theorem~\ref{thm:localconvsym} implies that after an additional $O(\log(1/\epsilon))$ number of iterations we find a solution $\hat X$ satisfying the accuracy requirement.
\end{proof}

\section{Reformulation of RIP-Constrained Optimization}

In this section, we will prove the following lemma that is used in Appendix~\ref{app:localproof} and Appendix~\ref{app:globalproof}, which is a generalization of Theorem~8 in \citet{ZSL2019}.

\begin{lemma}\label{lem:lmireform}
For every $a,b \geq 0$, the following two optimization problems
\begin{equation}\label{eq:opt}
\begin{aligned}
\min_{\delta,\mathbf H} \quad & \delta \\
\st \quad & \norm{\mathbf X^T\mathbf H\mathbf e} \leq a, \\
& 2I_r \otimes \mat_S(\mathbf H\mathbf e)+(1+\delta)\mathbf X^T\mathbf X \succeq -bI_{nr}, \\
& \text{$\mathbf H$ is symmetric and satisfies $\delta$-$\RIP_{2r}$},
\end{aligned}
\end{equation}
and
\begin{equation}\label{eq:lmi}
\begin{aligned}
\min_{\delta,\mathbf H} \quad & \delta \\
\st \quad & \begin{bmatrix}
I_{nr} & \mathbf X^T\mathbf H\mathbf e \\
(\mathbf X^T\mathbf H\mathbf e)^T & a^2
\end{bmatrix} \succeq 0, \\
& 2I_r \otimes \mat_S(\mathbf H\mathbf e)+(1+\delta)\mathbf X^T\mathbf X \succeq -bI_{nr}, \\
& (1-\delta)I_{n^2} \preceq \mathbf H \preceq (1+\delta)I_{n^2},
\end{aligned}
\end{equation}
have the same optimal value.
\end{lemma}
\begin{proof}
Let $\OPT(X,Z)$ be the optimal value of \eqref{eq:opt} and $\LMI(X,Z)$ be the optimal value of \eqref{eq:lmi}. Our goal is to prove that $\OPT(X,Z)=\LMI(X,Z)$ for given $X,Z \in \RR^{n \times r}$. Let $(v_1,\dots,v_n)$ be an orthogonal basis of $\RR^n$ such that $(v_1,\dots,v_d)$ spans the column spaces of both $X$ and $Z$. Note that $d \leq 2r$. Let $P \in \RR^{n \times d}$ be the matrix with the columns $(v_1,\dots,v_d)$ and $P_\perp \in \RR^{n \times (n-d)}$ be the matrix with the columns $(v_{d+1},\dots,v_n)$. Then,
\begin{gather*}
P^TP=I_d, \quad P_\perp^TP_\perp=I_{n-d}, \quad P_\perp^TP=0, \quad P^TP_\perp=0, \\
PP^T+P_\perp P_\perp^T=I_n, \quad PP^TX=X, \quad PP^TZ=Z.
\end{gather*}
Define $\mathbf P=P \otimes P$. Consider the auxiliary optimization problem
\begin{equation}\label{eq:aux}
\begin{aligned}
\min_{\delta,\mathbf H} \quad & \delta \\
\st \quad & \begin{bmatrix}
I_{nr} & \mathbf X^T\mathbf H\mathbf e \\
(\mathbf X^T\mathbf H\mathbf e)^T & a^2
\end{bmatrix} \succeq 0, \\
& 2I_r \otimes \mat_S(\mathbf H\mathbf e)+(1+\delta)\mathbf X^T\mathbf X \succeq -bI_{nr}, \\
& (1-\delta)I_{d^2} \preceq \mathbf P^T\mathbf H\mathbf P \preceq (1+\delta)I_{d^2},
\end{aligned}
\end{equation}
and denote its optimal value as the function $\overline\LMI(X,Z)$. Given an arbitrary symmetric matrix $\mathbf H \in \RR^{n ^2 \times n^2}$, if $\mathbf H$ satisfies the last constraint in \eqref{eq:lmi}, then it obviously satisfies $\delta$-$\RIP_{2r}$ and subsequently the last constraint in \eqref{eq:opt}. On the other hand, if $\mathbf H$ satisfies the last constraint in \eqref{eq:opt}, for every matrix $Y \in \RR^{d \times d}$ with $\mathbf Y=\vect(Y)$, since $\rank(PYP^T) \leq d \leq 2r$ and $\vect(PYP^T)=\mathbf P\mathbf Y$, by $\delta$-$\RIP_{2r}$ property, one arrives at
\[
(1-\delta)\norm{\mathbf Y}^2=(1-\delta)\norm{\mathbf P\mathbf Y}^2 \leq (\mathbf P\mathbf Y)^T\mathbf H\mathbf P\mathbf Y \leq (1+\delta)\norm{\mathbf P\mathbf Y}^2=(1+\delta)\norm{\mathbf Y}^2,
\]
which implies that $\mathbf H$ satisfies the last constraint in \eqref{eq:aux}. Moreover, since the first constraint in \eqref{eq:opt} and the first constraint in \eqref{eq:lmi} and \eqref{eq:aux} are equivalent, the above discussion implies that
\[
\LMI(X,Z) \geq \OPT(X,Z) \geq \overline\LMI(X,Z).
\]
Let
\[
\hat X=P^TX, \quad \hat Z=P^TZ.
\]
Lemma~\ref{lem:ineq1} and Lemma~\ref{lem:ineq2} to be stated later will show that
\[
\LMI(X,Z) \leq \LMI(\hat X,\hat Z) \leq \overline\LMI(X,Z),
\]
which gives $\OPT(X,Z)=\LMI(X,Z)$.
\end{proof}

Before stating Lemma~\ref{lem:ineq1} and Lemma~\ref{lem:ineq2} that were needed in the proof of Lemma~\ref{lem:lmireform}, we should first state a preliminary result below.

\begin{lemma}
Define $\hat{\mathbf e}$ and $\hat{\mathbf X}$ in the same way as $\mathbf e$ and $\mathbf X$, except that $X$ and $Z$ are replaced by $\hat X$ and $\hat Z$, respectively. Then, it holds that
\begin{align*}
\mathbf e&=\mathbf P\hat{\mathbf e}, \\
\mathbf X(I_r \otimes P)&=\mathbf P\hat{\mathbf X}, \\
\mathbf P^T\mathbf X&=\hat{\mathbf X}(I_r \otimes P)^T.
\end{align*}
\end{lemma}
\begin{proof}
Observe that
\begin{align*}
\mathbf e&=\vect(XX^T-ZZ^T)=\vect(P(\hat X\hat X^T-\hat Z\hat Z^T)P^T)=\mathbf P\hat{\mathbf e}, \\
\mathbf X(I_r \otimes P)\vect(\hat U)&=\mathbf X\vect(P\hat U)=\vect(X\hat U^TP^T+P\hat UX^T)=\vect(P(\hat X\hat U^T+\hat U\hat X^T)P^T)=\mathbf P\hat{\mathbf X}\vect(\hat U), \\
\hat{\mathbf X}(I_r \otimes P)^T\vect(U)&=\hat{\mathbf X}\vect(P^TU)=\vect(\hat XU^TP+P^TU\hat X^T)=\vect(P^T(XU^T+UX^T)P)=\mathbf P^T\mathbf X\vect(U),
\end{align*}
where $U \in \RR^{n \times r}$ and $\hat U \in \RR^{d \times r}$ are arbitrary matrices.
\end{proof}

\begin{lemma}\label{lem:ineq1}
The inequality $\LMI(\hat X,\hat Z) \geq \LMI(X,Z)$ holds.
\end{lemma}
\begin{proof}
Let $(\delta,\hat{\mathbf H})$ be an arbitrary feasible solution to the optimization problem defining $\LMI(\hat X,\hat Z)$. It is desirable to show that $(\delta,\mathbf H)$ with
\[
\mathbf H=\mathbf P\hat{\mathbf H}\mathbf P^T+(I_{n^2}-\mathbf P\mathbf P^T)
\]
is a feasible solution to the optimization problem defining $\LMI(X,Z)$, which directly proves the lemma. To this end, notice that
\[
\mathbf H-(1-\delta)I_{n^2}=\mathbf P(\hat{\mathbf H}-(1-\delta)I_{d^2})\mathbf P^T+\delta(I_{n^2}-\mathbf P\mathbf P^T),
\]
which is positive semidefinite because
\begin{align*}
I_{n^2}-\mathbf P\mathbf P^T&=(PP^T+P_\perp P_\perp^T) \otimes (PP^T+P_\perp P_\perp^T)-(PP^T) \otimes (PP^T) \\
&=(PP^T) \otimes (P_\perp P_\perp^T)+(P_\perp P_\perp^T) \otimes (PP^T)+(P_\perp P_\perp^T) \otimes (P_\perp P_\perp^T) \succeq 0.
\end{align*}
Similarly,
\[
\mathbf H-(1+\delta)I_{n^2} \preceq 0,
\]
and therefore the last constraint in \eqref{eq:lmi} is satisfied. Next, since
\[
\mathbf X^T\mathbf H\mathbf e=\mathbf X^T\mathbf H\mathbf P\hat{\mathbf e}=\mathbf X^T\mathbf P\hat{\mathbf H}\hat{\mathbf e}=(I_r \otimes P)\hat {\mathbf X}^T\hat{\mathbf H}\hat{\mathbf e},
\]
we have
\[
\norm{\mathbf X^T\mathbf H\mathbf e}^2=(\hat {\mathbf X}^T\hat{\mathbf H}\hat{\mathbf e})^T(I_r \otimes P^T)(I_r \otimes P)(\hat {\mathbf X}^T\hat{\mathbf H}\hat{\mathbf e})=\norm{\hat {\mathbf X}^T\hat{\mathbf H}\hat{\mathbf e}}^2,
\]
and thus the first constraint in \eqref{eq:lmi} is satisfied. Finally, by letting $W \in \RR^{d \times d}$ be the vector satisfying $\vect(W)=\hat{\mathbf H}\hat{\mathbf e}$, one can write
\[
\vect(PWP^T)=\mathbf P\vect(W)=\mathbf P\hat{\mathbf H}\hat{\mathbf e}.
\]
Hence,
\begin{align*}
2I_r \otimes \mat_S(\mathbf H\mathbf e)&=2I_r \otimes \mat_S(\mathbf H\mathbf P\hat{\mathbf e})=2I_r \otimes \mat_S(\mathbf P\hat{\mathbf H}\hat{\mathbf e})=I_r \otimes (P(W+W^T)P^T) \\
&=2I_r \otimes (P\mat_S(\hat{\mathbf H}\hat{\mathbf e})P^T)=2(I_r \otimes P)(I_r \otimes \mat_S(\hat{\mathbf H}\hat{\mathbf e}))(I_r \otimes P)^T.
\end{align*}
In addition,
\[
\mathbf X^T\mathbf X(I_r \otimes P)=\mathbf X^T\mathbf P\hat{\mathbf X}=(I_r \otimes P)\hat{\mathbf X}^T\hat{\mathbf X}.
\]
Therefore, by defining
\[
\mathbf S:=2I_r \otimes \mat_S(\mathbf H\mathbf e)+(1+\delta)\mathbf X^T\mathbf X+bI_{nr},
\]
we have
\begin{align*}
(I_r \otimes P)^T\mathbf S(I_r \otimes P)&=2I_r \otimes \mat_S(\hat{\mathbf H}\hat{\mathbf e})+(1+\delta)\hat{\mathbf X}^T\hat{\mathbf X}+bI_{dr} \succeq 0, \\
(I_r \otimes P_\perp)^T\mathbf S(I_r \otimes P_\perp)&=(1+\delta)(I_r \otimes P_\perp)^T\mathbf X^T\mathbf X(I_r \otimes P_\perp)+bI_{(n-d)r} \succeq 0, \\
(I_r \otimes P_\perp)^T\mathbf S(I_r \otimes P)&=0.
\end{align*}
Since the columns of $I_r \otimes P$ and $I_r \otimes P_\perp$ form a basis for $\RR^{nr}$, the above inequalities imply that $\mathbf S$ is positive semidefinite, and thus the second constraint in \eqref{eq:lmi} is satisfied.
\end{proof}

\begin{lemma}\label{lem:ineq2}
The inequality $\overline\LMI(X,Z) \geq \LMI(\hat X,\hat Z)$ holds.
\end{lemma}
\begin{proof}
The dual problem of the optimization problem defining $\LMI(\hat X,\hat Z)$ can be expressed as
\begin{equation}\label{eq:lmidual}
\begin{aligned}
\max_{\hat U_1,\hat U_2,\hat V,\hat G,\hat\lambda,\hat y} \quad & \tr(\hat U_1-\hat U_2)-\tr(\hat G)-a^2\hat\lambda-\langle\hat{\mathbf X}^T\hat{\mathbf X},\hat V\rangle-b\tr(\hat V) \\
\st \quad & \tr(\hat U_1+\hat U_2)+\langle\hat{\mathbf X}^T\hat{\mathbf X},\hat V\rangle=1, \\
&\left(\hat{\mathbf X}\hat y-\sum_{j=1}^r\vect(\hat V_{j,j})\right)\hat{\mathbf e}^T+\hat{\mathbf e}\left(\hat{\mathbf X}\hat y-\sum_{j=1}^r\vect(\hat V_{j,j})\right)^T=\hat U_1-\hat U_2, \\
& \begin{bmatrix}
\hat G & -\hat y \\
-\hat y^T & \hat \lambda
\end{bmatrix} \succeq 0, \\
&\hat U_1 \succeq 0, \quad \hat U_2 \succeq 0, \quad \hat V=\begin{bmatrix}
\hat V_{1,1} & \cdots & \hat V_{r,1} \\
\vdots & \ddots & \vdots \\
\hat V_{r,1}^T & \cdots & \hat V_{r,r}
\end{bmatrix} \succeq 0.
\end{aligned}
\end{equation}
Since
\[
\hat U_1=\frac{1-\mu\norm{\hat{\mathbf X}}^2}{2d^2}I_{d^2}-\frac{\mu r}{2}M, \quad \hat U_2=\frac{1-\mu\norm{\hat{\mathbf X}}^2}{2d^2}I_{d^2}+\frac{\mu r}{2}M, \quad \hat V=\mu I_{dr}, \quad \hat G=I_{dr}, \quad \hat\lambda=1, \quad \hat y=0,
\]
where
\[
M=\vect(I_d)\hat{\mathbf e}^T+\hat{\mathbf e}\vect(I_d)^T,
\]
is a strict feasible solution to the above dual problem \eqref{eq:lmidual} as long as $\mu>0$ is sufficiently small, Slater's condition implies that strong duality holds for the optimization problem defining $\LMI(\hat X,\hat Z)$. Therefore, we only need to prove that the optimal value of \eqref{eq:lmidual} is smaller than or equal to the optimal value of the dual of the optimization problem defining $\overline\LMI(X,Z)$ given by:
\begin{equation}\label{eq:lmidual2}
\begin{aligned}
\max_{U_1,U_2,V,G,\lambda,y} \quad & \tr(U_1-U_2)-\tr(G)-a^2\lambda-\langle\mathbf X^T\mathbf X,\mathbf V\rangle-b\tr(V) \\
\st \quad & \tr(U_1+U_2)+\langle\mathbf X^T\mathbf X,\mathbf V\rangle=1, \\
&\left(\mathbf Xy-\sum_{j=1}^r\vect(V_{j,j})\right)\mathbf e^T+\mathbf e\left(\mathbf Xy-\sum_{j=1}^r\vect(V_{j,j})\right)^T=\mathbf P(U_1-U_2)\mathbf P^T, \\
& \begin{bmatrix}
G & -y \\
-y^T & \lambda
\end{bmatrix} \succeq 0, \\
&U_1 \succeq 0, \quad U_2 \succeq 0, \quad V=\begin{bmatrix}
V_{1,1} & \cdots & V_{r,1} \\
\vdots & \ddots & \vdots \\
V_{r,1}^T & \cdots & V_{r,r}
\end{bmatrix} \succeq 0.
\end{aligned}
\end{equation}
The above claim can be verified by noting that given any feasible solution
\[
(\hat U_1,\hat U_2,\hat V,\hat G,\hat\lambda,\hat y)
\]
to \eqref{eq:lmidual}, the matrices
\begin{gather*}
U_1=\hat U_1, \quad U_2=\hat U_2, \quad V=(I_r \otimes P)\hat V(I_r \otimes P)^T, \\
\begin{bmatrix}
G & -y \\
-y^T & \lambda
\end{bmatrix}=\begin{bmatrix}
I_r \otimes P & 0 \\
0 & 1
\end{bmatrix}\begin{bmatrix}
\hat G & -\hat y \\
-\hat y^T & \hat\lambda
\end{bmatrix}\begin{bmatrix}
(I_r \otimes P)^T & 0 \\
0 & 1
\end{bmatrix}
\end{gather*}
form a feasible solution to \eqref{eq:lmidual2}, and both solutions have the same optimal value.
\end{proof}
\end{document}